\documentclass[11pt,a4paper]{article}
\usepackage[T1]{fontenc}
\usepackage{lmodern}
\usepackage[utf8]{inputenc}
\usepackage{amsthm}
\usepackage{amsmath,amssymb}
\usepackage{graphics}
\usepackage{todonotes}
\usepackage{lscape,graphicx}
\usepackage{enumitem}
\usepackage{amsfonts}
\usepackage[T1]{fontenc}
\usepackage{longtable}
\usepackage{authblk}
\usepackage{lipsum}
\usepackage{epsfig}
\usepackage{scalerel}
\usepackage{float}
\usepackage{hyperref}
\usepackage{comment}
\usepackage[normalem]{ulem}
\usepackage{multicol}
\usepackage{fancyhdr}
\usepackage[english]{babel}
\usepackage{mathrsfs}
\usepackage{tocloft}
\usepackage{xcolor}
\usepackage{titletoc}
\usepackage{mathtools}
\usepackage[toc,page]{appendix}

\addto{\captionsenglish}{}
\makeatletter
\def\thm@space@setup{%
  \thm@preskip=1cm plus .5cm minus .5cm
  \thm@postskip=.5cm plus .6cm minus .5cm % or whatever, if you don't want them to be equal
}

\makeatother

\newtheorem{thm}{Theorem}

\newtheorem{lma}{Lemma}

\newtheorem{cor}{Corollary}

\newtheorem{rmk}{Remark}

\newtheorem{app}{Application}
\numberwithin{thm}{section}
\numberwithin{lma}{section}
\numberwithin{dfn}{section}
\numberwithin{cor}{section}
\numberwithin{rmk}{section}
\numberwithin{prop}{section}
\numberwithin{app}{section}

\def\mfq{{\mathfrak \mfq}}

\def\mfm{{\mathfrak m}}
\def\mfc{{\mathfrak C}}
\def\mfa{{\mathfrak A}}
\def\mcm{{\mathcal M}}
\def\mcp{{\mathcal P}}
\def\mfp{{\mathfrak p}}
\def\mfq{{\mathfrak q}}

\newcommand*{\thmref}[1]{Theorem~\ref{#1}}

\newcommand*{\lmaref}[1]{Lemma~\ref{#1}}
\newcommand*{\rmkref}[1]{Remark~\ref{#1}}

\title{On the distribution of the total number of generators of $h$-free and $h$-full elements in an abelian monoid}

\date{}
\begin{document}

\author{Sourabhashis Das, Wentang Kuo, Yu-Ru Liu}

\maketitle 

\begin{abstract}
Let $\mfm$ be an element of an abelian monoid, with $\Omega(\mfm)$ denoting the total number of prime elements generating $\mfm$. We study the moments of $\Omega(\mfm)$ over subsets of $h$-free and $h$-full elements, establishing the normal order of $\Omega(\mathfrak{m})$ within these subsets. This work continues the study on the distribution of generalized arithmetic functions over $h$-free and $h$-full elements in abelian monoids as introduced in the authors' previous work \cite{dkl5}.
\end{abstract}

\section{Introduction}
\footnotetext{\textbf{2020 Mathematics Subject Classification: 11N80, 11K65, 20M32.}}\footnotetext{Keywords: Omega function, abelian monoids, the first and the second moments, $h$-free and $h$-full elements.}\footnotetext{The research of W. Kuo and Y.-R. Liu is supported by the Natural Sciences and Engineering Research Council of Canada (NSERC) grants.}
In their seminal work \cite{hardyram}, Hardy and Ramanujan demonstrated that, for a natural number $n$, the arithmetic functions $\omega(n)$, which counts the number of distinct prime divisors of $n$, and $\Omega(n)$, which counts the total number of prime divisors (counted with multiplicity), typically assume values close to $\log \log n$. This phenomenon, wherein an arithmetic function approximates a well-understood, non-negative function for all but a zero-density subset of the natural numbers, is referred to as its normal order. The original proof by Hardy and Ramanujan employed sophisticated tools from analytic number theory. Later, Tur\'an \cite{turan} furnished a more elementary and accessible argument by estimating the first and second moments of these functions. These foundational results have since been refined through more precise analyses of the associated error terms (see, for example, \cite{ElbGor} and \cite{saidak}).

The results of Hardy and Ramanujan, as well as those of Tur\'an, naturally lend themselves to further generalization. In \cite[Chapter 6, Theorem 2.4]{jk2}, first moment estimates for analogs of the functions $\omega(n)$ and $\Omega(n)$ were established in the setting of countably generated abelian monoids. Building upon these developments, the third author, in \cite{liuturan}, extended the classical framework by analyzing both the first and second moments of the analog of $\omega(n)$, thereby establishing its normal order. Following similar techniques to those in \cite{liuturan}, one can also derive the second moment for the analog of $\Omega(n)$, leading to the conclusion that it likewise exhibits a normal order in the monoid context. This line of inquiry was further developed in \cite{dkl5}, where the distribution of the $\omega(n)$-analog was examined over the subsets of $h$-free and $h$-full elements of an abelian monoid. In this article, we extend the investigation to the analog of $\Omega(n)$ over the same subsets. Unlike the previous study, the key distinction here lies in the inclusion of multiplicities of prime elements, which significantly influences the distribution of the function. To address this, we develop new identities involving sums over prime elements, which are established in the present work. Furthermore, we demonstrate several applications of our general theorems, particularly in the settings of ideals in number fields, effective divisors in global function fields, and effective zero-cycles on geometrically irreducible projective varieties.

We begin by introducing the framework of countably generated abelian monoids, along with the notions of $h$-free and $h$-full elements. We also recall several relevant results from the authors’ earlier work \cite{dkl5}, which will serve as a foundation for the present study.

Let $\mathcal{P}$ be a countable set of elements with a map
$$N: \mathcal{P} \rightarrow \mathbb{N} \backslash \{1\}, \quad \mfp \mapsto N(\mfp).$$
We call the elements of $\mathcal{P}$ as the \textit{prime elements}, and the map $N(\cdot)$ as the \textit{Norm map}. Let $\mathcal{M}$ be a \textit{free abelian monoid} generated by elements of $\mathcal{P}$. In other words, for each $\mfm \in \mathcal{M}$, we write
$$\mfm = \sum_{\mfp \in  \mathcal{P}} n_\mfp(\mfm) \mfp,$$
with $n_\mfp(\mfm) \in \mathbb{N} \cup \{ 0 \}$ and $n_\mfp(\mfm) =0$ for all but finitely many $\mfp$. We call $n_\mfp(\mfm)$ the \textit{multiplicity} of $\mfp$ in $\mfm$. We extend the norm map $N$ on $\mcm$ as the following:
\begin{align*}
    N : \mcm & \rightarrow \mathbb{N} \\
    \mfm = \sum_{\mfp \in \mathcal{P}} n_\mfp(\mfm) \mfp & \longmapsto N(\mfm) := \prod_{\mfp \in \mathcal{P}} N(\mfp)^{n_\mfp(\mfm)}. 
\end{align*}
Thus, $N$ can be extended to a monoid homomorphism from $(\mcm,+)$ to $(\mathbb{N},\cdot)$. Let $X$ be a countable subset of $\mathbb{Q}$ that contains the image $\text{Im}(N(\mcm))$ with an extra condition: if $x_1,x_2 \in X$, the fraction $x_1/x_2$ belongs to $X$, too. By \cite[Theorem 2]{liu}, without loss of generality, one can assume $X = \mathbb{Q}$ or $X = \{ q^{z}: z \in \mathbb{Z} \}$ where $q$ is a power of a natural number strictly greater than 1.

Given $\mcp$, $\mcm$, $X$, and for sufficiently large $x \in X$, we assume that the following condition hold: 
\begin{equation}\label{star}
    \mcm(x) := \sum\limits_{\substack{\mfm \in \mcm \\ N(\mfm) \leq x}} 1 = \kappa x + O(x^\theta), \quad \text{ for some } \kappa > 0 \text{ and } 0 \leq \theta < 1. \tag{$\star$}
\end{equation}
For each $\mfm \in \mcm$, we define
$$\omega(\mfm) = \sum_{\substack{\mfp \in \mathcal{P} \\ n_\mfp(\mfm) \geq 1}} 1,$$
the number of elements of $\mathcal{P}$ that generates $\mfm$, counted without multiplicity.

We also define
$$\Omega(\mfm) = \sum_{\substack{\mfp \in \mathcal{P} \\ n_\mfp(\mfm) \geq 1}} n_\mfp(\mfm),$$
the total number of elements of $\mathcal{P}$ that generates $\mfm$, counted with multiplicity.

% Let $K/\mathbb{Q}$ be a number field of degree $n_K = [K: \mathbb{Q}]$ and $\mathcal{O}_K$ be its ring of integers. Let $\mathcal{P}$ be the set of prime ideals of $\mathcal{O}_K$ and $\mathcal{M}$ be the set of ideals of $\mathcal{O}_K$. Let $N: \mathcal{M} \rightarrow \mathbb{N}$ be the standard norm map, i.e., $\mathfrak{m} \mapsto N(\mathfrak{m}) := |\mathcal{O}_K/ \mathfrak{m}|$.

For a non-zero element $\mathfrak{m} \in \mathcal{M}$, let the prime element factorization of $\mathfrak{m}$ be given as
\begin{equation}\label{factorization}
\mathfrak{m} = s_1 \mfp_1 + \cdots + s_r \mfp_r,
\end{equation}
where $\mfp_i'$s are its distinct prime elements and $s_i'$s are their respective non-zero multiplicities. Here, $\omega(\mfm) = r$, $\Omega(\mfm) = \sum_{i=1}^r s_i$, and 
$$N(\mfm) = N(\mfp_1)^{s_1} \cdots N(\mfp_r)^{s_r}.$$ 

Let $h \geq 2$ be an integer. We say $\mathfrak{m}$ is an $h$\textit{-free} element if $s_i \leq h-1$ for all $i \in \{1, \cdots, r\}$, and we say $\mathfrak{m}$ is an $h$\textit{-full} element if $s_i \geq h$ for all $i \in \{1, \cdots, r\}$
. Let $\mathcal{S}_h$ denote the set of $h$-free elements and $\mathcal{N}_h$ denote the set of $h$-full elements. The distribution of $h$-free elements in $\mcm$ is well-established. To demonstrate the result, we introduce the generalized $\zeta$-function which is an analog of the classical Riemann $\zeta$-function as the following:
\begin{equation}\label{gzeta}
    \zeta_\mcm(s) := \sum_{\substack{\mathfrak{m}}} \frac1{(N(\mathfrak{m}) )^s} 
= \prod_{\mfp} \Big( 1- N (\mfp)^{-s}\Big)^{-1}
\ \text{for } \Re (s) >1,
\end{equation}
where $\mathfrak{m}$ and $\mfp$ respectively range through the non-zero elements in $\mcm$ and the prime elements in $\mcp$. The absolute convergence of the series is explained in \cite{dkl5}.

Let $\mathcal{P}, \ \mcm$, and $X$ satisfy the Condition \eqref{star}. Let $x \in X$ and let $\mathcal{S}_h(x)$ denote the set of $h$-free elements with norm $N(\cdot)$ less than or equal to $x$. Since Condition \eqref{star} satisfies \cite[Chapter 4, Axiom A]{jk2}, thus, by \cite[Chapter 4, Proposition 5.5]{jk2}, we have:
\begin{equation}\label{hfreeidealcount}
    |\mathcal{S}_h(x)| = \frac{\kappa}{\zeta_\mcm(h)} x + O_h \big( R_{\mathcal{S}_h}(x) \big),
\end{equation}
where $|S|$ denotes the cardinality of the set $S$, and where
   \begin{equation}\label{RSh(x)}
    R_{\mathcal{S}_h}(x) = \begin{cases}
    x^\theta  & \text{ if } \frac{1}{h} < \theta, \\
    x^\theta (\log x) & \text{ if } \frac{1}{h} = \theta, \\
    x^{\frac{1}{h}} & \text{ if } \frac{1}{h} > \theta.
\end{cases}
    \end{equation}
\begin{rmk}\label{remark1} 
    In this paper, for convenience, we shall use $R_{\mathcal{S}_h}(x) \ll x^{\tau}$ for some $\tau \in [1/h,1)$, which is evident from the above result. For the first case in \eqref{RSh(x)}, $\tau = \theta$. For the second case, $\tau = \frac{1}{h} + \epsilon$ for any $\epsilon \in \left( 0, \frac{h-1}{h} \right)$, and for the third case, $\tau = 1/h$. 
\end{rmk}
% \textcolor{red}{\begin{thm}\label{hfreeideals}
%     Let $x > 1$ be any real number and $h \geq 2$ be any integer. Let $\mathcal{S}_h(x)$ denote the set of $h$-free ideals with norm $N(\cdot)$ less than or equal to $x$. We have
%     $$|\mathcal{S}_h(x)| = \frac{\kappa}{\zeta_\mcm(h)} x + O_h \big( R_{\mathcal{S}_h}(x) \big),$$
%     where
%     \begin{equation*}
%     R_{\mathcal{S}_h}(x) = \begin{cases}
%     x^{1 - \frac{2}{n_{\scaleto{K}{3pt}}+1}}  & \text{ if } \frac{1}{h} < 1 - \frac{2}{n_{\scaleto{K}{3pt}} + 1}, \\
%     x^{1 - \frac{2}{n_{\scaleto{K}{3pt}}+1}} (\log x) & \text{ if } \frac{1}{h} = 1 - \frac{2}{n_{\scaleto{K}{3pt}} + 1}, \\
%     x^{\frac{1}{h}} & \text{ if } \frac{1}{h} > 1 - \frac{2}{n_{\scaleto{K}{3pt}} + 1}.
% \end{cases}
%     \end{equation*}
% \end{thm}}
% \begin{rmk}\label{remark1}
%     In this work, for convenience, we shall use $R_{\mathcal{S}_h}(x) \ll x^{\tau}$ for some $\tau < 1$  which is evident from the above result. Also, note that, if $n_K = 1$, then $K = \mathbb{Q}$, $\zeta_\mathbb{Q}$ is the classical Riemann $\zeta$-function, and $R_{\mathcal{S}_h}(x) = x^{1/h}$ which matches the distribution of $h$-free natural numbers (see \cite[(4)]{jala}).
% \end{rmk}
Let $\gamma_h$ be the convergent Euler product given as 
\begin{equation}\label{gammahk}
    \gamma_{\scaleto{h}{4.5pt}} = \gamma_{\scaleto{h,\mcm}{5.5pt}} := \prod_\mfp \left( 1 + \frac{N(\mfp) - N(\mfp)^{1/h}}{N(\mfp)^2 \left( N(\mfp)^{1/h} - 1 \right)}\right).
\end{equation}
Let $\mathcal{N}_h(x)$ denote the set of $h$-full elements with norm $N(\cdot)$ less than or equal to $x$. For the distribution of $h$-full elements, from \cite[Theorem 1.1]{dkl5}, we have:
$$|\mathcal{N}_h(x)| = \kappa \gamma_{\scaleto{h}{4.5pt}} x^{1/h} + O_h \big( R_{\mathcal{N}_h}(x) \big),$$
where
\begin{equation}\label{E2(x)}
    R_{\mathcal{N}_h}(x) = \begin{cases}
        x^{\frac{\theta}{h}} & \text{ if } \frac{h}{h+1} < \theta, \\
        x^{\frac{1}{h+1}} (\log x) & \text{ if } \frac{h}{h+i} = \theta  \text{ for some } i \in \{ 1, \ldots, h-1 \}, \\
        x^{\frac{1}{h+1}} & \text{ if } \frac{h}{h+1} > \theta \text{ and } \frac{h}{h+i} \neq \theta \text{ for any } i \in \{ 1, \cdots, h-1 \}.
    \end{cases} 
\end{equation}
\begin{rmk}\label{remark2}
    In this paper, again for convenience, we shall use $R_{\mathcal{N}_h}(x) \ll x^{\upsilon/h}$ for some $\upsilon \in [h/(h+1),1)$. For the first case in \eqref{E2(x)}, $\upsilon = \theta$. For the second case, $\upsilon = \frac{h}{h+1} + \epsilon$ for any $\epsilon \in \left( 0,\frac{1}{h+1} \right)$, and for the third case, $\upsilon = \frac{h}{h+1}$. 
\end{rmk}

% \begin{rmk}\label{remark2}
%     In this work, again for convenience, we shall use $R_{\mathcal{N}_h}(x) \ll x^{\xi/h}$ for some $\xi < 1$  which is evident from \thmref{hfullideals}. Also, note that, for $n_K =1$, $R_{\mathcal{N}_h}(x) = x^{1/(h+1)}$ which matches the distribution of $h$-full numbers from the work of Ivi\'c and Shiu (see \cite[Lemma 1]{is}).
% \end{rmk}
In \cite{dkl5}, the aforementioned distribution results were employed to analyze the moments of $\omega(\mfm)$ over the subsets of $h$-free and $h$-full elements. In particular, in \cite[Theorem 1.2]{dkl5}, we established the first and the second moments of $\omega(\mfm)$ over $h$-free elements, and in \cite[Theorem 1.3]{dkl5}, we established the first and the second moments of $\omega(\mfm)$ over $h$-full elements.

In the present work, we investigate the distribution of $\Omega(\mfm)$ and establish results analogous to those previously obtained for $\omega(\mfm)$.

Let $x \in X$. Let $\mathfrak{A}$ and $\mathfrak{B}$ be constants defined as
\begin{equation}\label{A}
    \mfa:= \lim_{x \rightarrow \infty} \left( \sum_{\substack{\mfp \in \mathcal{P} \\ N(\mfp) \leq x}} \frac{1}{N(\mfp)} - \log \log x \right),
\end{equation}
and
\begin{equation}\label{B}
    \mathfrak{B} := \begin{cases}
        - \pi^2/6  & \text{if } X = \mathbb{Q} \text{ and} \\
         (\log \log \mfq)^2 - \pi^2/6 & \text{if } X = \{ \mfq^{z} : z \in \mathbb{Z} \}.
    \end{cases}
\end{equation}
The existence of the constant $\mathfrak{A}$ is explained in \cite[Lemma 2]{liuturan}. For instance, if $\mcm = \mathbb{N}$, 
\begin{equation*}
\mathfrak{A} = \gamma + \sum_p \left( \log \left( 1 - \frac{1}{p} \right) + \frac{1}{p} \right) \approx 0.2615,
\end{equation*}
with $\gamma \approx 0.57722$, the Euler-Mascheroni constant, and where the sum runs over all primes.

Let $\mathfrak{C}_3$, $\mathfrak{C}_3'$, and $\mathfrak{C}_4$\footnote{We reserve the subscripts 1 and 2 for the constants used in our previous work \cite{dkl5} to maintain clarity in this extension of those results.} be three new constants (dependent on $h$) defined as
\begin{equation}\label{def-C3}
    \mathfrak{C}_3 := \mfa + \sum_\mfp \frac{N(\mfp)^h -h N(\mfp)^2 + hN(\mfp) - 1}{N(\mfp)(N(\mfp)-1)(N(\mfp)^h - 1)},
\end{equation}
\begin{small}
\begin{align}\label{def-C3'}
    & \mathfrak{C}_3' \notag \\
    & := \mfa  + \notag \\
    & \sum_\mfp \frac{N(\mfp)^h(3N(\mfp)-1) + (N(\mfp)-1)^2 - N(\mfp)(h^2 N(\mfp)^2 + (-2h^2+2h+1) N(\mfp) + (h-1)^2)}{N(\mfp)(N(\mfp)-1)^2(N(\mfp)^h - 1)},
\end{align}
\end{small}
and
\begin{equation}\label{def-C4}
    \mathfrak{C}_4 := \mathfrak{C}_3^2 + \mathfrak{C}_3' + \mathfrak{B} - \sum_\mfp  \left( \frac{N(\mfp)^{h} -h N(\mfp) + h - 1}{(N(\mfp)-1)(N(\mfp)^h - 1)} \right)^2.
\end{equation}

For the distribution of $\Omega(\mfm)$ over $h$-free elements, we prove:
\begin{thm}\label{hfreeOmega}
Let $\mathcal{P}, \mcm$, and $X$ satisfy the Condition \eqref{star}. Let $x \in X$ and $h \geq 2$ be an integer. Let $\mathcal{S}_h(x)$ be the set of $h$-free elements with norm $N(\cdot)$ less than or equal to $x$. Then, we have
$$\sum_{\mathfrak{m} \in \mathcal{S}_h(x)} \Omega(\mfm) = \frac{\kappa}{\zeta_\mcm(h)} x \log \log x + 
\frac{\kappa \mfc_3}{\zeta_\mcm(h)} x
%\textcolor{red}{R(F,\Gamma)} x 
+ O_h \left( \frac{x}{\log x}\right),$$
and
\begin{align*}
\sum_{\mfm \in \mathcal{S}_h(x)}  \Omega^2(\mfm) 
& = \frac{\kappa}{\zeta_\mcm(h)}  x (\log \log x)^2 + \frac{\kappa(2 \mathfrak{C}_3 + 1)}{\zeta_\mcm(h)} x \log \log x + \frac{\kappa \mathfrak{C}_4}{\zeta_\mcm(h)} x + O_h \left( \frac{x \log \log x}{\log x}\right),
\end{align*}
where $\zeta_\mcm(s)$ is defined in \eqref{gzeta}.
\end{thm} 

We define the constants (dependent on $h$)
\begin{equation}\label{def-D3}
    \mathfrak{D}_3 = h (\mfa - \log h) + \sum_\mfp \frac{h(N(\mfp)-N(\mfp)^{1-(1/h)}-N(\mfp)^{1/h} + 1)+N(\mfp)}{N(\mfp) (N(\mfp)^{1/h} - 1) (N(\mfp) - N(\mfp)^{1-(1/h)} +1)},
\end{equation}
\begin{small}
\begin{align}\label{def-D3'}
    & \mathfrak{D}_3' \notag \\
    & = h^2(\mfa - \log h) + \notag \\
    & \sum_\mfp \left( \frac{(2h^2 + 2h - 1)N(\mfp)^{\frac{1 + h}{h}} - (1 + h)^2N(\mfp)^{\frac{2 + h}{h}} - h^2(N(\mfp) - N(\mfp)^{\frac{1}{h}} + 2N(\mfp)^{\frac{2}{h}} - N(\mfp)^{\frac{3}{h}})}{N(\mfp)(-N(\mfp)^{\frac{1 + h}{h}} - N(\mfp)^{\frac{1}{h}} + N(\mfp))(-1 + N(\mfp)^{\frac{1}{h}})^2} \right),
\end{align}
\end{small}
and
\begin{equation}\label{def-D4}
    \mathfrak{D}_4 := \mathfrak{D}_3^2+ \mathfrak{D}_3' + h^2 \mathfrak{B} - \sum_\mfp \left( \frac{h(N(\mfp)^{\frac{1}{h}}-1)+1}{(N(\mfp)^{\frac{1}{h}}-1) (N(\mfp) - N(\mfp)^{1-\frac{1}{h}} +1)} \right)^2.
\end{equation}

For the distribution of $\Omega(\mfm)$ over $h$-full elements, we prove:
\begin{thm}\label{hfullOmega}
Let $\mathcal{P}, \mcm$, and $X$ satisfy the Condition \eqref{star}. Let $x \in X$ and $h \geq 2$ be an integer. Let $\mathcal{N}_h(x)$ be the set of $h$-full elements with norm $N(\cdot)$ less than or equal to $x$. Then, we have
\begin{align*}
\sum_{\substack{\mfm \in \mathcal{N}_h(x)}} \Omega(\mfm) & =   h \kappa \gamma_{h} x^{1/h} \log \log x + \mathfrak{D}_3 \kappa \gamma_{h} x^{1/h} + O_h \left( \frac{x^{1/h}}{\log x} \right),
\end{align*}
and
\begin{align*}
     \sum_{\substack{\mfm \in \mathcal{N}_h(x)}} \Omega^2(\mfm)
     & = h^2 \kappa \gamma_{h} x^{1/h} (\log \log x)^2 + (2\mathfrak{D}_3 + h) h \kappa \gamma_{h} x^{1/h} \log \log x + \mathfrak{D}_4 \kappa \gamma_{h} x^{1/h} \\
     & \hspace{.5cm} + O_h \left( \frac{x^{1/h} \log \log x}{\log x} \right),
\end{align*}
where $\gamma_h$ is defined in \eqref{gammahk}.
\end{thm} 

To establish \thmref{hfreeOmega} and \thmref{hfullOmega}, we adapt the novel counting method introduced in \cite{dkl2} and subsequently utilized in \cite{dkl5} and \cite{dkl4}. This approach is based on decomposing the sum over elements of the monoid into sums indexed by prime elements of bounded norm. For each such prime element, the corresponding summand involves a sum over monoid elements with positive multiplicity at that prime. While the underlying strategy remains largely the same, the analysis of $\Omega(\mfm)$ necessitates additional refinements to account for the contribution of prime multiplicities. In particular, we make use of the expressions \eqref{sumkp^k/h} and \eqref{need-k^2} to isolate the main terms and effectively control the error terms in the distribution of $\Omega(\mfm)$ over $h$-full elements. Additionally, we derive a new identity for sums over prime elements, given in \lmaref{sumplogx/N(mfp)^2}, which plays a key role in bounding the error terms arising in our analysis.

We also note that, in the case $\mcm = \mathbb{N}$, \thmref{hfreeOmega} and \thmref{hfullOmega} refine the results of \cite[Theorems 1 \& 2]{jala} by explicitly determining the coefficient of $x$ in the former and improving the error term in the latter. These enhancements were also discussed in our earlier work \cite[Theorems 1.4 \& 1.5]{dkl2}\footnote{We take this opportunity to correct a mistake in \cite[Theorems 1.4 \& 1.5]{dkl2}: the constants $C_4$ in Theorem 1.4 and $B_4$ in Theorem 1.5 should be replaced by the corrected constants $\mathfrak{C}_4$ and $\mathfrak{D}_4$, respectively, as presented in the current article.}.
    
We now proceed to demonstrate that \thmref{hfreeOmega} and \thmref{hfullOmega} are applicable to a variety of abelian monoids previously studied in the literature. Several illustrative applications are listed below. For additional details regarding the structure of these monoids and certain key properties, such as identity \eqref{star}, we refer the readers to \cite[Section 5]{dkl5}.

\begin{app}\textbf{For $\boldsymbol{h}$-free and $\boldsymbol{h}$-full ideals in number fields.}
    Let $K/\mathbb{Q}$ be a number field of degree $n_K = [K: \mathbb{Q}]$ and $\mathcal{O}_K$ be its ring of integers. Let $\mathcal{P}$ be the set of prime ideals of $\mathcal{O}_K$ and $\mathcal{M}$ be the set of ideals of $\mathcal{O}_K$. Let the norm map be $N: \mathcal{M} \rightarrow \mathbb{N}$ be the standard norm map, i.e., $\mathfrak{m} \mapsto N(\mathfrak{m}) := |\mathcal{O}_K/ \mathfrak{m}|$. Let $X = \mathbb{Q}$.

Landau in \cite[Satz 210]{Landau} proved that
\begin{equation*}
    \sum_{\substack{\mfm \in \mcm \\ N(\mfm) \leq x}} 1 = \kappa_K x + O \left( x^{1 - \frac{2}{n_{\scaleto{K}{3pt}}+1}}\right),
\end{equation*}
where $\kappa_K$ is a field invariant, and which satisfies Condition \eqref{star} with $\kappa = \kappa_K$ and $\theta = 1 - \frac{2}{n_{K}+1}$. Thus, \thmref{hfreeOmega} and \thmref{hfullOmega} give the distribution of $\Omega(\mfm)$ over $h$-free and $h$-full ideals.
\end{app}
\begin{rmk}
    For the special case $K = \mathbb{Q}$, we have $n_K=1$ and $\kappa_K = 1$. Consequently, our results yield the distribution of $\Omega(n)$ over $h$-free and $h$-full numbers, thereby refining \cite[Theorems 1 \& 2]{jala}, as previously discussed.
\end{rmk}

\begin{app}\textbf{For $\boldsymbol{h}$-free and $\boldsymbol{h}$-full effective divisors in global function fields.}

    Let $q$ be a prime power and $\mathbb{F}_q$ be the finite field with $q$ elements. Let $K/\mathbb{F}_q$ be a global function field. Let $G_K$ be its genus and $C_K$ be its class number. A prime $\mfp$ in $K$ is a discrete valuation ring $R$ with maximal ideal $P$ such that $P \subset R$ and the quotient field of $R$ is $K$. The degree of $\mfp$, denoted as $\deg \mfp$, is defined as the dimension of $R/P$ over $\mathbb{F}_q$, which is finite. Let $\mcp$ be the set of all primes in $K$. Let $\mcm$ be the free abelian monoid generated by $\mcp$. Let the norm map $N: \mathcal{M} \rightarrow \mathbb{N}$ be the $q$-power map defined as $\mathfrak{m} \mapsto N(\mfm) := q^{\deg \mathfrak{m}}$. Let $X = \{ q^z : z \in \mathbb{Z} \}$. 

By \cite[Lemma 5.8 \& Corollary 4 to Theorem 5.4]{mr}, for sufficiently large $n$, we obtain
$$\sum_{N(\mfm) \leq q^n} 1 = \frac{C_K}{q^{G_K}} \left( \frac{q}{q-1} \right)^2 q^n + O(n).$$
This satisfies Condition \eqref{star} with $\kappa = \frac{C_K}{q^{G_K}} \left( \frac{q}{q-1} \right)^2$ and $\theta = \epsilon$ for any $\epsilon \in (0,1)$. Thus, \thmref{hfreeOmega} and \thmref{hfullOmega} respectively give the distribution of $\Omega(\mfm)$ over $h$-free and $h$-full effective divisors in a global function field.
\end{app}
\begin{rmk}
For the special case when $K = \mathbb{F}_q(x)$, whose genus and class number are 0 and 1 respectively, we can consider the abelian monoid $A = \mathbb{F}_q[x]$, the ring of monic polynomials in one variable over $\mathbb{F}_q$. Thus, we obtain the distribution of $\Omega(\mfm)$ over $h$-free and $h$-full polynomials over finite fields, which will be equivalent to the ones studied by Lal\'in and Zhang \cite[Theorems 1.1 \& 1.4]{lz}. 
\end{rmk}

\begin{app}\textbf{For $\boldsymbol{h}$-free and $\boldsymbol{h}$-full effective \texorpdfstring{$\boldsymbol{0}$}{}-cycles in geometrically irreducible projective varieties of dimension \texorpdfstring{$\boldsymbol{r}$}{}.}

Let $q$ be a prime power and $\mathbb{F}_q$ be the finite field with $q$ elements. Let $r$ be a positive integer. Let $V/\mathbb{F}_q$ be a geometrically irreducible projective variety of dimension $r$. Let $\mcp$ be the set of closed points of $V/\mathbb{F}_q$, which is in bijection with the set of orbits of $V/\mathbb{F}_q$ under the action of $\textnormal{Gal}(\bar{\mathbb{F}}_q/\mathbb{F}_q)$ (see \cite[Proposition 6.9]{lorenzini}). For each $\mfp \in \mcp$, we define the degree of $\mfp$, $\deg \mfp$, to be the length of the corresponding orbit. Let $\mcm$ be the free abelian monoid generated by $\mcp$. Let the norm map $N: \mathcal{M} \rightarrow \mathbb{N}$ be the $q^r$-power map defined as $\mathfrak{m} \mapsto N(\mfm) := q^{r \deg \mathfrak{m}}$. Let $X = \{ q^{rz} : z \in \mathbb{Z} \}$. 

In \cite[Remark 1 of Section 4]{liuturan}, the third author proved that
$$\sum_{N(\mfm) \leq q^{rn}} 1 = \kappa' \left( \frac{q^r}{q^r-1} \right) q^{rn} + O_r \left( n \cdot q^{(r-1)n} \right),$$
where $\kappa'$ is some positive constant defined explicitly in \cite[Lemma 7 of Section 4]{liuturan}. This satisfies Condition \eqref{star} with $\kappa = \kappa' \left( \frac{q^r}{q^r-1} \right)$ and $\theta = \epsilon$ for any $\epsilon \in (1-1/r,1)$. Thus, \thmref{hfreeOmega} and \thmref{hfullOmega} respectively give the distribution of $\Omega(\mfm)$ over $h$-free and $h$-full effective 0-cycles in a geometrically irreducible projective variety of dimension $r$.
\end{app}

In \cite{hardyram}, Hardy and Ramanujan strengthened the classical distribution result for $\omega(n)$ by proving that $\omega(n)$ has the normal order $\log \log n$ over natural numbers. In \cite[Section 22.11]{hw}, one can find another proof of this result using the variance of $\omega(n)$ (see \cite[(22.11.7)]{hw}). In \cite[Corollary 1]{liuturan}, using this method of variance, the third author provides an analog of this result over any abelian monoid $\mathcal{M}$. 
% More precisely, she showed that for any $\epsilon > 0$, the number of elements of $\mcm$ with norm $N(\cdot)$ less than or equal to $x$ that do not satisfy the inequality
%  $$(1-\epsilon) \log \log N(\mathfrak{m}) \leq \omega(\mathfrak{m}) \leq (1+ \epsilon) \log \log N(\mathfrak{m})$$
% is $o(x)$ as $x \rightarrow \infty$. This shows that $\omega(\mathfrak{m})$ has normal order $\log \log N(\mathfrak{m})$ over $\mathcal{M}$. 
Moreover, in \cite{dkl5}, the authors showed that $\omega(\mfm)$ has normal order $\log \log N(\mfm)$ over the set of $h$-free elements and over the set of $h$-full elements. In this work, we extend the normal order results over these subsets to $\Omega(\mfm)$.

First, we define the normal order of a function over a subset of $\mathcal{M}$. Let $\mathfrak{S} \subseteq \mathcal{M}$ and $\mathfrak{S}(x)$ denote the set of elements belonging to $\mathfrak{S}$ and with norm $N(\cdot)$ less than or equal to $x$. Let $|\mathfrak{S}(x)|$ denote the cardinality of $\mathfrak{S}(x)$. Let $f, F : \mathfrak{S} \rightarrow \mathbb{R}_{\geq 0}$ be two functions such that $F$ is non-decreasing, i.e., if $\mathfrak{m}_1, \mathfrak{m}_2 \in \mathcal{M}$ with $N(\mfm_1) \leq N(\mfm_2)$, then $F(\mfm_1) \leq F(\mfm_2)$. Then, $f(\mathfrak{m})$ is said to have normal order $F(\mathfrak{m})$ over $\mathfrak{S}$ if for any $\epsilon > 0$, the number of $\mathfrak{m} \in \mathfrak{S}(x)$ that do not satisfy the inequality
$$(1-\epsilon) F(\mathfrak{m}) \leq f(\mathfrak{m}) \leq (1+ \epsilon) F(\mathfrak{m})$$
is $o(|\mathfrak{S}(x)|)$ as $x \rightarrow \infty$. 

% In \cite{turan}, Tur\'an proved that
% $$\sum_{n \leq x} (\omega(n) -\log \log n)^2 \ll x \log \log x,$$
% and then used this result to prove that for any $\delta > 0$ and $A >0$, the number of $n \leq x$ that satisfy the inequality
% $$|\omega(n) - \log \log n| > A (\log \log n)^{\frac{1}{2} + \delta}$$
% is $o(x)$ as $x \rightarrow \infty$. 
% From this, one can again deduce that $\omega(n)$ has normal order $\log \log n$.
% If $\mathfrak{S} = \mathcal{M}$, then the proof of normal order of $\omega(\mathfrak{m})$ being $\log \log N(\mathfrak{m})$ over $\mathcal{M}$ follows directly from the classical case. In particular, we can show that for any $\epsilon > 0$, the number of ideals with norm $N(\cdot)$ less than or equal to $x$ that do not satisfy the inequality
% $$(1-\epsilon) \log \log N(\mathfrak{m}) \leq \omega(\mathfrak{m}) \leq (1+ \epsilon) \log \log N(\mathfrak{m})$$
% is $o(x)$ as $x \rightarrow \infty$.
The set of $h$-free elements has a positive density $1/\zeta_\mcm(h)$ in $\mathcal{M}$. Thus, the proof of normal order of $\Omega(\mathfrak{m})$ over $h$-free elements being $\log \log N(\mathfrak{m})$ follows from the classical case. In particular, one can establish the following:
\begin{cor}
For any $\epsilon > 0$, the number of $\mathfrak{m} \in \mathcal{S}_h(x)$ that do not satisfy the inequality
$$(1-\epsilon) \log \log N(\mathfrak{m}) \leq \Omega(\mathfrak{m}) \leq (1+ \epsilon) \log \log N(\mathfrak{m})$$
is $o(|\mathcal{S}_h(x)|)$ as $x \rightarrow \infty$.
\end{cor}
On the other hand, the set of $h$-full elements has density zero in $\mathcal{M}$ and thus does not follow directly as the previous two result. However, writing an $h$-full element $\mathfrak{m}$ as $\mathfrak{m} = h\mathfrak{r}_0 + \mathfrak{r}_1$ where $\mathfrak{r}_0$ is the sum of all distinct prime element factors $\mfp$ of $\mathfrak{m}$ with $n_\mfp(\mfm) = h$, one can use the classical result on $\Omega(\mathfrak{r}_0)$ to establish the normal order of $\Omega(\mathfrak{m})$ over $h$-full elements. Additionally, working similarly to \cite[Proof of Theorem 1.3]{dkl2}, using the method of variance, we can establish the following result: 
\begin{cor}
For any $\epsilon > 0$, the number of $\mathfrak{m} \in \mathcal{N}_h(x)$ that do not satisfy the inequality
$$(1-\epsilon) h \log \log N(\mathfrak{m}) \leq \Omega(\mathfrak{m}) \leq (1+ \epsilon) h \log \log N(\mathfrak{m})$$
is $o(|\mathcal{N}_h(x)|)$ as $x \rightarrow \infty$.
\end{cor}

In 1940, Erd\H{o}s and Kac \cite{ErdosKac} showed that for $\mcm = \mathbb{N}$, the properly normalized function $\omega(\mfm)$ converges in distribution to the standard normal law. In \cite{liu}, the third author generalized this result to all abelian monoids considered within this framework. This line of research was further advanced in \cite{dkl6}, where the authors established Gaussian distribution results for both $\omega(\mfm)$ and $\Omega(\mfm)$ over the subsets of $h$-free and $h$-full elements. Notably, the article also introduced a unified approach for studying a broad class of additive functions.

\section{Essential Lemmas}

In this section, we list all the lemmas required to study the first and the second moment of $\Omega(\mfm)$ over $h$-free and over $h$-full elements. First, we recall the following results from \cite{dkl5} regarding sums over prime elements necessary for the study:
\begin{lma}\cite[Lemma 2.1]{dkl5}\label{gpnt}
Let $\mathcal{P}, \mcm$, and $X$ satisfy the Condition \eqref{star}. Let $x \in X$. Then
\begin{equation*}
    \Pi(x) := \sum\limits_{\substack{\mfp \in \mcp \\ N(\mfp) \leq x}} 1 = O \left(\frac{x}{\log x} \right).
\end{equation*}
\end{lma}
\begin{lma}\cite[Lemma 2.2, Parts (1),(3),(4) \& (6)]{dkl5}\label{boundnm}
   Let $\mathcal{P}, \mcm$, and $X$ satisfy the Condition \eqref{star}. Let $x \in X$ and $\alpha$ be a real number. We have
   \begin{enumerate}
       \item If $0 \leq \alpha < 1$, 
       $$\sum_{\substack{\mfp \in \mathcal{P} \\ N(\mfp) \leq x}} \frac{1}{N(\mfp)^\alpha} = O_\alpha \left( \frac{x^{1- \alpha}}{\log x} \right).$$
       \item If $\alpha > 1$, then
       $$\sum_{\substack{\mfp \in \mathcal{P} \\ N(\mfp) \leq x}} \frac{1}{N(\mfp)^\alpha} = O_\alpha (1).$$
       \item As a generalization of Mertens' theorem, we have
       $$\sum_{\substack{\mfp \in \mathcal{P} \\ N(\mfp) \leq x}} \frac{1}{N(\mfp)} = \log \log x + \mfa + O \left( \frac{1}{\log x} \right),$$
       where $\mfa$ some constant that depends only on $\mcp$.
       \item If $\alpha > 1$, then
       $$\sum_{\substack{\mfp \in \mathcal{P} \\ N(\mfp) \geq x}} \frac{1}{N(\mfp)^\alpha} = O \left( \frac{1}{(\alpha-1)x^{\alpha-1} (\log x)} \right).$$
   \end{enumerate}
\end{lma}
\begin{lma}\cite[Lemma 2.3]{dkl5}\label{sumplogp}
    Let $\mathcal{P}, \mcm$, and $X$ satisfy the Condition \eqref{star}. Let $x \in X$. If $X = \mathbb{Q}$, we have
    $$\sum_{N(\mfp) \leq x/2} \frac{1}{N(\mfp) \log (x/N(\mfp))} = O \left( \frac{\log \log x}{\log x} \right),$$
    and if $X = \{ q^{z} : z \in \mathbb{Z} \}$, we have
    $$\sum_{N(\mfp) \leq x/q} \frac{1}{N(\mfp) \log (x/N(\mfp))} = O \left( \frac{\log \log x}{\log x} \right).$$
\end{lma}
\begin{lma}\cite[Lemma 2.4]{dkl5}\label{sumnwpx/2} 
  Let $\mathcal{P}, \mcm$, and $X$ satisfy the Condition \eqref{star}. Let $x \in X$. If $X = \mathbb{Q}$, we have
    $$\sum_{\substack{\mfp \\ N(\mfp) \leq x/2}}  \frac{1}{N(\mfp)} \log \log \frac{x}{N(\mfp)} = (\log \log x)^2 + \mfa \log \log x + \mathfrak{B} +  O \left( \frac{\log \log x}{\log x} \right),$$
and if $X = \{ q^{z} : z \in \mathbb{Z} \}$, we have
  $$\sum_{\substack{\mfp \\ N(\mfp) \leq x/q}}  \frac{1}{N(\mfp)} \log \log \frac{x}{N(\mfp)} = (\log \log x)^2 + \mfa \log \log x + \mathfrak{B} +  O \left( \frac{\log \log x}{\log x} \right),$$
where $\mfa$ and $\mathfrak{B}$ are defined in \eqref{A} and \eqref{B} respectively.
\end{lma}
\begin{lma}\cite[Lemma 2.5]{dkl5}\label{gsaidak}
  Let $\mathcal{P}, \mcm$, and $X$ satisfy the Condition \eqref{star}. Let $\mfp$ and $\mfq$ denote prime elements. Let $x \in X$. Then, we have
\begin{equation*}
    \sum_{\substack{\mfp, \mfq \\ N(\mfp) N(\mfq) \leq x}} \frac{1}{N(\mfp) N(\mfq)} = (\log \log x)^2 + 2 \mfa \log \log x + \mfa^2 + \mathfrak{B} + O \left( \frac{\log \log x}{\log x} \right),
\end{equation*}
where $\mfa$ and $\mathfrak{B}$ are defined in \eqref{A} and \eqref{B} respectively.
\end{lma}
Next, we establish a new result to study the involvement of prime elements appearing in our work.
\begin{lma}\label{sumplogx/N(mfp)^2}
   Let $\mathcal{P}, \mcm$, and $X$ satisfy the Condition \eqref{star}. Let $\mfp$ denote a prime element. Let $x \in X$. If $X = \mathbb{Q}$, we have
     $$\sum_{N(\mfp) \leq x/2} \frac{1}{N(\mfp) \log^2 (x/N(\mfp))} = O \left( \frac{1}{\log x} \right),$$
   and if $X = \{ q^{z} : z \in \mathbb{Z} \}$, we have
    $$\sum_{N(\mfp) \leq x/q} \frac{1}{N(\mfp) \log^2 (x/N(\mfp))} = O \left( \frac{1}{\log x} \right).$$
\end{lma}
\begin{proof}
    Let $X = \mathbb{Q}$. We partition the interval $[1,x/2]$ into subintervals of the form $I_j = [e^{j-1}, e^{j}]$. Over this interval, by \lmaref{gpnt}, we have
    $$\sum_{N(\mfp) \in I_j} \frac{1}{N(\mfp) \log^2 (x/N(\mfp))} \ll \frac{1}{\log^2(x/e^j)} \sum_{N(\mfp) \in I_j} \frac{1}{N(\mfp)} \ll \frac{1}{j (\log x - j)^2}.$$
  Putting the above together with the fact that 
  \begin{align*}
      \sum_{\substack{\mfp \\ e^{\lfloor \log(x/2) \rfloor} < N(\mfp) \leq x/2}} \frac{1}{N(\mfp) \log^2 (x/N(\mfp))} \ll \frac{1}{\log^2 2} \sum_{\substack{\mfp \\ e^{\lfloor \log(x/2) \rfloor} < N(\mfp) \leq x/2}} \frac{1}{N(\mfp)} 
      % & \ll \frac{1}{x} \sum_{\substack{N(\mfp) \leq x/2}} 1 \\
      & \ll \frac{1}{\log x},
  \end{align*}
  where in the second last inequality we use \lmaref{gpnt}, we obtain
  \begin{align}\label{need1}
      \sum_{N(\mfp) \leq x/2} \frac{1}{N(\mfp) \log^2 (x/N(\mfp))} & = \sum_{j=1}^{\lfloor \log(x/2) \rfloor} \sum_{N(\mfp) \in I_j} \frac{1}{N(\mfp) \log^2 (x/N(\mfp))}  + O \left( \frac{1}{\log x} \right)\notag \\
      % & =
      % \sum_{j=1}^{\lfloor \log(x/2) \rfloor} \frac{1}{j \log^2(x/e^j)} + O \left( \frac{1}{\log x} \right) \notag\\
      & = \sum_{j=1}^{\lfloor \log(x/2) \rfloor} \frac{1}{j (\log x - j)^2} + O \left( \frac{1}{\log x} \right).
  \end{align}
  Note that
  $$\frac{1}{j(\log -j)^2} = \frac{1}{(\log x)^2} \left( \frac{1}{j} + \frac{1}{\log x -j} \right) +\frac{1}{(\log x)(\log x - j)^2}.$$
  Moreover,
  $$\sum_{j=1}^{\lfloor \log(x/2) \rfloor} \frac{1}{j} 
  % \ll \int_{t=1}^{\log (x/2)} \frac{dt}{t} 
  \ll \log \log x,
  \quad
  \sum_{j=1}^{\lfloor \log(x/2) \rfloor} \frac{1}{\log x -j} 
  % \ll \int_{t=1}^{\log (x/2)} \frac{dt}{\log x - t} 
  \ll \log \log x, \quad \text{and } \sum_{j=1}^{\lfloor \log(x/2) \rfloor} \frac{1}{(\log x -j)^2} 
  % \ll \int_{t=1}^{\log (x/2)} \frac{dt}{(\log x - t)^2} 
  \ll 1.$$
  As a result,
  $$\sum_{j=1}^{\lfloor \log(x/2) \rfloor} \frac{1}{j (\log x - j)^2} \ll \frac{\log \log x}{(\log x)^2} + \frac{1}{\log x} \ll \frac{1}{\log x}.$$
  Putting the above back into \eqref{need1} completes the proof when $X = \mathbb{Q}$. For the case $X = \{ q^{z} : z \in \mathbb{Z} \}$, we partition the interval $[1,x/q]$ into subintervals of the form $I_j = [e^{j-1}, e^{j}]$ and repeat the above steps to complete the proof.
\end{proof}

Next, we record some distribution results for $h$-free and $h$-full elements which are co-prime to a given set of prime elements. 

\begin{lma}\label{hfreeidealrestrict}\cite[Lemma 3.1]{dkl5}
    Let $h \geq 2$ and $r \geq 1$ be integers. Let $\ell_1,\ldots,\ell_r$ be fixed distinct prime elements and $\mathcal{S}_{h,\ell_1,\ldots,\ell_r}(x)$ denote the set of $h$-free elements $\mfm$ with norm $N(\mfm) \leq x$ and with $n_{\ell_i}(\mfm) = 0$ for all $i \in \{1,\cdots, r\}$. Then, we have
    $$|\mathcal{S}_{h,\ell_1,\ldots,\ell_r}(x)| = \prod_{i=1}^r \left( \frac{N(\ell_i)^h - N(\ell_i)^{h-1}}{N(\ell_i)^h - 1} \right) \frac{\kappa}{\zeta_\mcm(h)} x + O_{h,r} \left( R_{\mathcal{S}_h}(x) \right),$$
    where $\zeta_\mcm(\cdot)$ is the generalized $\zeta$-function defined in \eqref{gzeta}, and $R_{\mathcal{S}_h}(x)$ is defined in \eqref{RSh(x)}.
\end{lma}
\begin{lma}\label{hfullidealsrestrict}\cite[Lemma 4.2]{dkl5}
Let $x \in X$ and let $h \geq 2$ be an integer. Let $\ell_1, \cdots, \ell_r$ be fixed prime elements and $\mathcal{N}_{h,\ell_1,\cdots, \ell_r}(x)$ denote the set of $h$-full elements with norm $N(\mfm) \leq x$ and with $n_{\ell_i}(\mfm) = 0$ for all $i \in \{1,\cdots, r\}$. Then, we have
$$|\mathcal{N}_{h,\ell_1,\cdots, \ell_r}(x)| = \prod_{i=1}^r \frac{\kappa \gamma_{\scaleto{h}{4.5pt}}}{\left( 1+ \frac{N(\ell_i)^{-1}}{1-N(\ell_i)^{-1/h}} \right)} x^{1/h} + O_{h,r} \big( R_{\mathcal{N}_h}(x) \big),$$
where $\gamma_{\scaleto{h}{4.5pt}}$ is defined in \eqref{gammahk},
% \begin{equation}\label{gammahk}
%     \gamma_{\scaleto{h}{4.5pt}} = \gamma_{\scaleto{h,K}{5.5pt}} := \prod_\mfp \left( 1 + \frac{N(\mfp) - N(\mfp)^{1/h}}{N(\mfp)^2 \left( N(\mfp)^{1/h} - 1 \right)}\right),
% \end{equation}
and where $R_{\mathcal{N}_h}(x)$ is defined in \eqref{E2(x)}.
\end{lma}

In the following sections, we will always assume that $\mathcal{P}, \mcm$, and $X$ satisfy Condition \eqref{star}.
\section{The moments of \texorpdfstring{$\Omega(n)$}{} over \texorpdfstring{$h$}{}-free numbers}

This section proves \thmref{hfreeOmega} by establishing the first and the second moments of $\Omega(\mfm)$ over $h$-free elements.

\begin{proof}[\textbf{Proof of \thmref{hfreeOmega}}]
Using $\mfm = k \mfp + \mathfrak{y}$ with $n_\mfp(\mathfrak{y}) = 0$ when $n_\mfp(\mfm) = k$, and $r_\mfp = \min \left\{ h-1, \left\lfloor \log x / \log N(\mfp) \right\rfloor \right\}$, we obtain
\begin{equation*}
    \sum_{\mfm \in \mathcal{S}_h(x)} \Omega(\mfm) = \sum_{\mfm \in \mathcal{S}_h(x)} \sum_{\substack{\mfp \\ n_\mfp(\mfm) = k}} k = \sum_{N(\mfp) \leq x} \sum_{k=1}^{r_\mfp} \sum_{\substack{\mfm \in \mathcal{S}_h(x) \\ n_\mfp(\mfm) = k}} k = \sum_{N(\mfp) \leq x} \sum_{k=1}^{r_\mfp} k \sum_{\substack{\mathfrak{y} \in \mathcal{S}_h(x/N(\mfp)^k) \\ n_\mfp(\mathfrak{y}) = 0 }} 1.
\end{equation*}
Now, using \lmaref{hfreeidealrestrict} for a single prime element $\mfp$ to the above, part 1 of \lmaref{boundnm} and \rmkref{remark1}, we obtain
\begin{align}\label{anomega2Omega}
    \sum_{\mfm \in \mathcal{S}_h(x)} \Omega(n) & = \sum_{N(\mfp) \leq x} \sum_{k=1}^{r_\mfp} k \left( \frac{N(\mfp)^{h} - N(\mfp)^{h-1}}{N(\mfp)^k(N(\mfp)^h - 1)} \right) \frac{\kappa x}{\zeta_\mcm(h)} + O_h \left( x^{\tau} \sum_{N(\mfp) \leq x} 
 \frac{1}{N(\mfp)^{\tau}} \right) \notag \\
 & = \sum_{N(\mfp) \leq x} \sum_{k=1}^{r_\mfp} k \left( \frac{N(\mfp)^{h} - N(\mfp)^{h-1}}{N(\mfp)^k(N(\mfp)^h - 1)} \right) \frac{\kappa x}{\zeta_\mcm(h)} + O_h \left(  \frac{x}{\log x} \right).
\end{align}
To complete the proof, we only need to estimate the main term above. Recall the set
\begin{equation*}
    R:= \left\{ \mfp \ : \ N(\mfp) \leq x, \ \left\lfloor \frac{\log x}{\log N(\mfp)}  \right\rfloor < h-1 \right\}.
\end{equation*}
Using the bound $\lfloor x \rfloor \geq x - 1$ and \lmaref{gpnt}, we can write
\begin{align}\label{caser=h-1Omega}
    & \sum_{N(\mfp) \leq x} \sum_{k=1}^{r_\mfp} k \left( \frac{N(\mfp)^{h} - N(\mfp)^{h-1}}{N(\mfp)^k(N(\mfp)^h - 1)} \right) \notag \\
    & = \sum_{N(\mfp) \leq x} \sum_{k=1}^{h-1} k \left( \frac{N(\mfp)^{h} - N(\mfp)^{h-1}}{N(\mfp)^k(N(\mfp)^h - 1)} \right) -  \sum_{\substack{N(\mfp) \leq x \\ \mfp \in R}} \sum_{k=\left\lfloor \frac{\log x}{\log N(\mfp)}  \right\rfloor + 1}^{h-1} k \left( \frac{N(\mfp)^{h} - N(\mfp)^{h-1}}{N(\mfp)^k(N(\mfp)^h - 1)} \right) \notag \\
    % & = \sum_{N(\mfp) \leq x} \sum_{k=1}^{h-1} \left( \frac{N(\mfp)^{h} - N(\mfp)^{h-1}}{N(\mfp)^k(N(\mfp)^h - 1)} \right) + O \left( \sum_{N(\mfp) \leq x}  \frac{1}{N(\mfp)^{\frac{\log x}{\log N(\mfp)}}}\right) \notag  \\
    % & = \sum_{N(\mfp) \leq x} \sum_{k=1}^{h-1} \left( \frac{N(\mfp)^{h} - N(\mfp)^{h-1}}{N(\mfp)^k(N(\mfp)^h - 1)} \right) + O \left( \frac{1}{x} \sum_{N(\mfp) \leq x}  1 \right) \notag  \\
    & = \sum_{N(\mfp) \leq x} \sum_{k=1}^{h-1} k \left( \frac{N(\mfp)^{h} - N(\mfp)^{h-1}}{N(\mfp)^k(N(\mfp)^h - 1)} \right) + O_h \left( \frac{1}{\log x} \right).
\end{align}
Now, using 
\begin{align*}
    \frac{N(\mfp)^{h} - N(\mfp)^{h-1}}{N(\mfp)^h - 1} \sum_{k=1}^{h-1} \frac{k}{N(\mfp)^k}  
    % & =  \frac{N(\mfp)^h-h N(\mfp)+h-1}{(N(\mfp)-1)(N(\mfp)^h-1)} \\
    & = \frac{1}{N(\mfp)} + \frac{N(\mfp)^h -h N(\mfp)^2 + h N(\mfp) - 1}{N(\mfp)(N(\mfp)-1)(N(\mfp)^h - 1)},
\end{align*} 
the formula for the reciprocal of prime elements given in  part 3 of \lmaref{boundnm}, the formula for $\mathfrak{C}_3$ given in \eqref{def-C3}, and using part 4 of \lmaref{boundnm} with $\alpha = 2$, we obtain  
\begin{align*}
    \sum_{N(\mfp) \leq x} \sum_{k=1}^{h-1} k \left( \frac{N(\mfp)^{h} - N(\mfp)^{h-1}}{N(\mfp)^k(N(\mfp)^h - 1)} \right)
    % & = \sum_{N(\mfp) \leq x} \left( \frac{1}{N(\mfp)} - \frac{N(\mfp) - 1}{N(\mfp)(N(\mfp)^h - 1)} \right) \notag \\
    % & = \log \log x + B_1 + O \left( \frac{1}{\log x} \right) - \sum_{N(\mfp)} \frac{N(\mfp) - 1}{N(\mfp)(N(\mfp)^h - 1)} + O_h \left( \frac{1}{x^{h-1} \log x} \right) \notag \\
    & = \log \log x + \mathfrak{C}_3 + O_h \left( \frac{1}{\log x} \right).
\end{align*}
Combining the above with \eqref{anomega2Omega} and \eqref{caser=h-1Omega} completes the first part of the proof.

Next, we prove the second-moment estimate for $\Omega(\mfm)$ over $h$-free elements. Notice that
\begin{equation}\label{mainpartOmega}
    \sum_{\substack{\mfm \in \mathcal{S}_h(x)}} \Omega^2(n) = \sum_{\substack{\mfm \in \mathcal{S}_h(x)}} \left( \sum_{\substack{\mfp \\ n_\mfp(\mfm) = k}} k \right)^2 = \sum_{\substack{\mfm \in \mathcal{S}_h(x)}} \sum_{\substack{\mfp \\ n_\mfp(\mfm) = k}} k^2 + \sum_{\substack{ \mfm \in \mathcal{S}_h(x)}} \sum_{\substack{\mfp,\mfq \\ n_\mfp(\mfm) = k, n_\mfq(\mfm) = l, \mfp \neq \mfq }} (k l),
\end{equation}
where $\mfp$ and $\mfq$ above denote prime elements. Let $r_\mfp = \min \left\{ h-1, \left\lfloor \frac{\log x}{\log N(\mfp)} \right\rfloor \right\}$. To estimate the first sum on the right-hand side above, we rewrite the sum as
\begin{equation*}
 \sum_{\substack{\mfm \in \mathcal{S}_h(x)}} \sum_{\substack{\mfp \\ n_\mfp(\mfm)=k}} k^2 = \sum_{N(\mfp) \leq x} \sum_{k=1}^{r_\mfp} \sum_{\substack{\mfm \in \mathcal{S}_h(x) \\ n_\mfp(\mfm)=k}} k^2 = \sum_{N(\mfp) \leq x} \sum_{k=1}^{r_\mfp} k^2 \sum_{\substack{\mathfrak{y} \in \mathcal{S}_h(x/N(\mfp)^k) \\ n_\mfp(\mathfrak{y}) = 0 }} 1.   
\end{equation*}
Again, similar to \eqref{anomega2Omega} and \eqref{caser=h-1Omega}, using \lmaref{hfreeidealrestrict} for a single prime $\mfp$, part 1 of \lmaref{boundnm}, and \rmkref{remark1}, we obtain
\begin{align*}
    \sum_{N(\mfp) \leq x} \sum_{k=1}^{r_\mfp} k^2 \sum_{\substack{\mathfrak{y} \in \mathcal{S}_h(x/N(\mfp)^k) \\ (\mathfrak{y},N(\mfp)) = 0 }} 1 & = \sum_{N(\mfp) \leq x} \sum_{k=1}^{h-1} k^2 \left( \frac{N(\mfp)^{h} - N(\mfp)^{h-1}}{N(\mfp)^k(N(\mfp)^h - 1)} \right) \frac{\kappa x}{\zeta_\mcm(h)} + O_h \left(  \frac{x}{\log x} \right).
\end{align*}
Note that
\begin{align*}
   & \sum_{k=1}^{h-1} k^2 \left( \frac{N(\mfp)^{h} - N(\mfp)^{h-1}}{N(\mfp)^k(N(\mfp)^h - 1)} \right) \\ 
   & = \frac{1}{N(\mfp)} \\
   & + \frac{N(\mfp)^h(3p-1) + (N(\mfp)-1)^2 - N(\mfp)(h^2 N(\mfp)^2 + (-2h^2+2h+1) N(\mfp) + (h-1)^2)}{N(\mfp)(N(\mfp)-1)^2(N(\mfp)^h - 1)}.
\end{align*}
Thus, using parts 3 and 4 with $\alpha = 2$ of \lmaref{boundnm}, we obtain
\begin{align*}
    \sum_{k=1}^{h-1} k^2 \left( \frac{N(\mfp)^{h} - N(\mfp)^{h-1}}{N(\mfp)^k(N(\mfp)^h - 1)} \right) = \log \log x + \mathfrak{C}_3' + O_h \left( \frac{1}{\log x} \right),
\end{align*}
where $\mathfrak{C}_3'$ is given in \eqref{def-C3'}. 

Combining the last four results, we obtain
\begin{equation}\label{bound_sum1}
    \sum_{\substack{\mfm \in \mathcal{S}_h(x)}} \sum_{\substack{\mfp \\ n_\mfp(\mfm) = k}} k^2 =  \frac{\kappa}{\zeta_\mcm(h)} x \log \log x + \frac{\mathfrak{C}_3' \kappa}{\zeta_\mcm(h)} x + O_h \left( \frac{x}{\log x} \right).
\end{equation}

Next, for the second sum in \eqref{mainpartOmega}, we rewrite the sum, use \lmaref{hfreeidealrestrict} for two primes $\mfp$ and $\mfq$ and \rmkref{remark1} to obtain
\begin{align}\label{mainpart2Omega}
    & \sum_{\substack{\mfm \in \mathcal{S}_h(x)}} \sum_{\substack{\mfp,\mfq \\ n_\mfp(\mfm) = k, n_\mfq(\mfm) = l,  \mfp \neq \mfq }}  (k l) \notag \\
    & = \sum_{\substack{\mfp, \mfq \\ \mfp \neq \mfq, N(\mfp) N(\mfq) \leq x}} \sum_{k=1}^{r_\mfp} \sum_{l=1}^{r_\mfq} (k l) \sum_{\substack{\mfm \in  \mathcal{S}_h(x/N(\mfp)^k N(\mfq)^l) \\ n_\mfp(\mfm) = n_\mfq(\mfm)= 0}} 1 \notag \\
    %& = \sum_{\substack{N(\mfp),\mfq \\ N(\mfp) \neq \mfq, pq \leq x}} \sum_{k=1}^{r_\mfp} \sum_{l=1}^{r_\mfq} \left( \left( \frac{N(\mfp)^{h} - N(\mfp)^{h-1}}{N(\mfp)^k(N(\mfp)^h - 1)} \right) \left( \frac{\mfq^{h} - \mfq^{h-1}}{\mfq^l(\mfq^h - 1)} \right) \frac{x}{\zeta(h)} + O_h \left( \frac{x^{1/h}}{N(\mfp)^{k/h} \mfq^{l/h}} \right) \right)\\
    & = \sum_{\substack{\mfp,\mfq \\ \mfp \neq \mfq, N(\mfp) N(\mfq) \leq x}} \sum_{k=1}^{r_\mfp} \sum_{l=1}^{r_\mfq} \Bigg( \left( \frac{k(N(\mfp)^{h} - N(\mfp)^{h-1})}{N(\mfp)^{k}(N(\mfp)^h - 1)} \right) \left( \frac{l(N(\mfq)^{h} - N(\mfq)^{h-1})}{N(\mfq)^{l}(N(\mfq)^h - 1)} \right) \frac{x}{\zeta_\mcm(h)} \notag \\
    & \hspace{.5cm} + O_h \left( \frac{x^{\tau}}{N(\mfp)^{k \tau} N(\mfq)^{l \tau}} \right) \Bigg). 
    %\sum_{\substack{n \leq x/(N(\mfp)^k \mfq^l) \\ n \in \mathcal{S}_h \\ (N(\mfp),n) = (\mfq,n) = 1}} 1 - \sum_{\substack{N(\mfp) \\ N(\mfp) \leq x^{1/2}}} \sum_{k=1}^{r_\mfp} \sum_{l=1}^{r_\mfp} \sum_{\substack{n \leq x/N(\mfp)^{k+l}\\ n \in \mathcal{S}_h \\ (N(\mfp),n) = 1}} 1 \\
\end{align}
By part 1 of \lmaref{boundnm} with $\alpha = \tau$ and \lmaref{sumplogp}, we obtain
\begin{align*}
   x^{\tau} \sum_{\substack{\mfp, \mfq \\ N(\mfp) N(\mfq) \leq x}}  \frac{1}{N(\mfp) N(\mfq)^{\tau}} 
   % & = x^{1/h} \sum_{\substack{N(\mfp) \leq x/2}}\frac{1}{N(\mfp)^{1/h}} \sum_{\substack{N(\mfq) \leq x/N(\mfp)}} \frac{1}{N(\mfq)^{1/h}} \notag \\
    & \ll_{h} x \sum_{\substack{N(\mfp) \leq x/2}} \left(  \frac{1}{N(\mfp) \log(x/N(\mfp))}\right) 
   \ll_{h} \frac{x \log \log x}{\log x}.
\end{align*}
Thus, the error term in \eqref{mainpart2Omega} is estimated as 
\begin{equation}\label{mainpart3Omega}
    \sum_{\substack{\mfp,\mfq \\ \mfp \neq \mfq, N(\mfp) N(\mfq) \leq x}} \sum_{k=1}^{r_\mfp} \sum_{l=1}^{r_\mfq} \frac{x^{\tau}}{N(\mfp)^{k \tau} N(\mfq)^{l \tau}}
    \ll x^{\tau} \sum_{\substack{\mfp, \mfq \\ N(\mfp) N(\mfq) \leq x}}  \frac{1}{N(\mfp)^\tau N(\mfq)^{\tau}} 
    \ll_h \frac{x \log \log x}{\log x}.
\end{equation}
Thus, it remains to estimate the main term, in particular, the term
$$ \sum_{\substack{\mfp,\mfq \\ \mfp \neq \mfq, N(\mfp) N(\mfq) \leq x}} \sum_{k=1}^{r_\mfp} \sum_{l=1}^{r_\mfq} \left( \frac{k(N(\mfp)^{h} - N(\mfp)^{h-1})}{N(\mfp)^{k}(N(\mfp)^h - 1)} \right) \left( \frac{l(N(\mfq)^{h} - N(\mfq)^{h-1})}{N(\mfq)^{l}(N(\mfq)^h - 1)} \right).$$
Consider again the set
$$R:= \left\{ \mfp  \ : \ N(\mfp) \leq x, \ \left\lfloor \frac{\log x}{\log N(\mfp)} \right\rfloor < h-1 \right\}.$$
Using the definition of $r_\mfp$ and the symmetry of $\mfp$ and $\mfq$, we can rewrite 
\begin{align}\label{partsOmega}
    & \sum_{\substack{\mfp,\mfq \\ \mfp \neq \mfq, N(\mfp) N(\mfq) \leq x}} \sum_{k=1}^{r_\mfp} \sum_{l=1}^{r_\mfq} \left( \frac{k(N(\mfp)^{h} - N(\mfp)^{h-1})}{N(\mfp)^{k}(N(\mfp)^h - 1)} \right) \left( \frac{l(N(\mfq)^{h} - N(\mfq)^{h-1})}{N(\mfq)^{l}(N(\mfq)^h - 1)} \right) \notag \\
    & = \sum_{\substack{\mfp,\mfq \\ \mfp \neq \mfq, N(\mfp) N(\mfq) \leq x}} \sum_{k=1}^{h-1} \sum_{l=1}^{h-1} \left( \frac{k(N(\mfp)^{h} - N(\mfp)^{h-1})}{N(\mfp)^{k}(N(\mfp)^h - 1)} \right) \left( \frac{l(N(\mfq)^{h} - N(\mfq)^{h-1})}{N(\mfq)^{l}(N(\mfq)^h - 1)} \right) -  2 I_1 + I_2,
\end{align}
where
$$I_1 = \sum_{\substack{\mfp,\mfq \\ \mfp \neq \mfq, N(\mfp) N(\mfq) \leq x \\ \mfp \in R}} \sum_{k=\left\lfloor \frac{\log x}{\log N(\mfp)} \right\rfloor + 1}^{h-1} \sum_{l=1}^{h-1} \left( \frac{k(N(\mfp)^{h} - N(\mfp)^{h-1})}{N(\mfp)^{k}(N(\mfp)^h - 1)} \right) \left( \frac{l(N(\mfq)^{h} - N(\mfq)^{h-1})}{N(\mfq)^{l}(N(\mfq)^h - 1)} \right),$$
% $$I_2 = \sum_{\substack{\mfp,\mfq \\ \mfp \neq \mfq, N(\mfp) N(\mfq) \leq x \\ \mfq \in R}} \sum_{k=1}^{h-1} \sum_{l=\left\lfloor \frac{\log x}{\log \mfq} \right\rfloor + 1}^{h-1} \left( \frac{k(N(\mfp)^{h} - N(\mfp)^{h-1})}{N(\mfp)^{k}(N(\mfp)^h - 1)} \right) \left( \frac{l(N(\mfq)^{h} - N(\mfq)^{h-1})}{N(\mfq)^{l}(N(\mfq)^h - 1)} \right),$$
$$I_2 = \sum_{\substack{\mfp,\mfq \\ \mfp \neq \mfq, N(\mfp) N(\mfq) \leq x \\ \mfp, \mfq \in R}}  \sum_{k=\left\lfloor \frac{\log x}{\log N(\mfp)} \right\rfloor + 1}^{h-1} \sum_{l=\left\lfloor \frac{\log x}{\log \mfq} \right\rfloor + 1}^{h-1} \left( \frac{k(N(\mfp)^{h} - N(\mfp)^{h-1})}{N(\mfp)^{k}(N(\mfp)^h - 1)} \right) \left( \frac{l(N(\mfq)^{h} - N(\mfq)^{h-1})}{N(\mfq)^{l}(N(\mfq)^h - 1)} \right).$$
The sums $I_1$ and $I_2$ contribute to the error term. In fact, for $X = \mathbb{Q}$, using $\lfloor x \rfloor \geq x-1$, \lmaref{gpnt}, and \lmaref{sumplogp}, we estimate
\begin{small}
$$I_1 \ll_h \sum_{\substack{\mfp,\mfq \\ \mfp \neq \mfq, N(\mfp) N(\mfq) \leq x \\ \mfp \in R}} \frac{1}{N(\mfp)^{\frac{\log x}{\log N(\mfp)}}} \frac{1}{N(\mfq)} \ll_h \sum_{N(\mfq) \leq x/2} \frac{1}{N(\mfq)}  \left( \frac{1}{x} \sum_{N(\mfp) \leq x/N(\mfq)} 1 \right) 
% \ll_h \sum_{\mfq \leq x/2} \frac{1}{\mfq^2 \log(x/\mfq)} 
\ll_h \frac{\log \log x}{\log x}.$$
\end{small}
Similarly, for $I_2$, using $\lfloor x \rfloor \geq x-1$, \lmaref{gpnt}, and \lmaref{sumplogp} again, we have
$$I_2 \ll_h \sum_{\substack{\mfp,\mfq \\ \mfp \neq \mfq, N(\mfp) N(\mfq) \leq x \\ \mfp,\mfq \in R}} \frac{1}{N(\mfp)^{\frac{\log x}{\log N(\mfp)}}} \frac{1}{N(\mfq)^{\frac{\log x}{\log N(\mfq)}}} \ll_h \frac{1}{x} \sum_{N(\mfq) \leq x/2} \left( \frac{1}{x} \sum_{N(\mfp) \leq x/N(\mfq)} 1 \right) \ll_h \frac{\log \log x}{x \log x}.$$
For $X = \{ q^z : z \in \mathbb{Z} \}$, a similar result as above can be proved using $x/q$ instead of $x/2$.

We next estimate the main term in \eqref{partsOmega}. First, using
$$\frac{N(\mfp)^{h} - N(\mfp)^{h-1}}{N(\mfp)^h - 1} \sum_{k=1}^{h-1} \frac{k}{N(\mfp)^k}  =  \frac{N(\mfp)^h-h N(\mfp)+h-1}{(N(\mfp)-1)(N(\mfp)^h-1)},$$ 
and a similar result for $\mfq$, we rewrite the sum as
\begin{align}\label{part1Omega}
    & \sum_{\substack{\mfp,\mfq \\ \mfp \neq \mfq, N(\mfp) N(\mfq) \leq x}} \sum_{k=1}^{h-1} \sum_{l=1}^{h-1} \left( \frac{k(N(\mfp)^{h} - N(\mfp)^{h-1})}{N(\mfp)^{k}(N(\mfp)^h - 1)} \right) \left( \frac{l(\mfq^{h} - \mfq^{h-1})}{\mfq^{l}(\mfq^h - 1)} \right) \notag \\
    % & = \sum_{\substack{\mfp,\mfq \\ \mfp \neq \mfq, N(\mfp) N(\mfq) \leq x}} \left( \frac{N(\mfp)^{h} -h N(\mfp) + h - 1}{(N(\mfp)-1)(N(\mfp)^h - 1)} \right) \left( \frac{N(\mfq)^{h} -h N(\mfq) + h - 1}{(N(\mfq)-1)(N(\mfq)^h - 1)} \right) \notag \\
    & = \sum_{\substack{\mfp,\mfq \\ N(\mfp) N(\mfq) \leq x}} \left( \frac{N(\mfp)^{h} -h N(\mfp) + h - 1}{(N(\mfp)-1)(N(\mfp)^h - 1)} \right) \left( \frac{N(\mfq)^{h} -h N(\mfq) + h - 1}{(N(\mfq)-1)(N(\mfq)^h - 1)} \right) \notag \\
    & \hspace{.5cm} - \sum_{\substack{N(\mfp) \leq x^{1/2}}} \left( \frac{N(\mfp)^{h} -h N(\mfp) + h - 1}{(N(\mfp)-1)(N(\mfp)^h - 1)} \right)^2.
\end{align}
The second sum above is estimated using part 4 of \lmaref{boundnm} as
\begin{align}\label{part2Omega}
    \sum_{N(\mfp) \leq x^{1/2}} \left( \frac{N(\mfp)^{h} -h N(\mfp) + h - 1}{(N(\mfp)-1)(N(\mfp)^h - 1)} \right)^2 
    % & = \sum_{\mfp} \left( \frac{N(\mfp)^{h} -h N(\mfp) + h - 1}{(N(\mfp)-1)(N(\mfp)^h - 1)} \right)^2 + \sum_{\substack{N(\mfp) > x^{1/2}}} \left( \frac{N(\mfp)^{h} -h N(\mfp) + h - 1}{(N(\mfp)-1)(N(\mfp)^h - 1)} \right)^2 \notag \\
    & = \sum_{\mfp} \left( \frac{N(\mfp)^{h} -h N(\mfp) + h - 1}{(N(\mfp)-1)(N(\mfp)^h - 1)} \right)^2 + O \left( \frac{1}{x^{1/2} \log x}\right).
\end{align}
For the first sum on the right-hand side in \eqref{part1Omega}, using 
$$\frac{N(\mfp)^h-h N(\mfp)+h-1}{(N(\mfp)-1)(N(\mfp)^h-1)} = \frac{1}{N(\mfp)} + \frac{N(\mfp)^h -h N(\mfp)^2 + h N(\mfp) - 1}{N(\mfp)(N(\mfp)-1)(N(\mfp)^h - 1)},$$ 
and the symmetry in $\mfp$ and $\mfq$, we have
\begin{align}\label{part3Omega}
    & \sum_{\substack{\mfp,\mfq \\ N(\mfp) N(\mfq) \leq x}} \left( \frac{N(\mfp)^{h} -h N(\mfp) + h - 1}{(N(\mfp)-1)(N(\mfp)^h - 1)} \right) \left( \frac{N(\mfq)^{h} -h N(\mfq) + h - 1}{(N(\mfq)-1)(N(\mfq)^h - 1)} \right) \notag \\
    & = \sum_{\substack{\mfp,\mfq \\ N(\mfp) N(\mfq) \leq x}} \frac{1}{N(\mfp) N(\mfq)} + 2 \sum_{\substack{\mfp,\mfq \\ N(\mfp) N(\mfq) \leq x}} \frac{1}{N(\mfp)} \frac{N(\mfq)^h -h N(\mfq)^2 + h N(\mfq) - 1}{N(\mfq)(N(\mfq)-1)(N(\mfq)^h - 1)} \notag \\
    & + \sum_{\substack{\mfp,\mfq \\ N(\mfp) N(\mfq) \leq x}} \left( \frac{N(\mfp)^h -h N(\mfp)^2 + h N(\mfp) - 1}{N(\mfp)(N(\mfp)-1)(N(\mfp)^h - 1)} \right) \left( \frac{N(\mfq)^h -h N(\mfq)^2 + h N(\mfq) - 1}{N(\mfq)(N(\mfq)-1)(N(\mfq)^h - 1)} \right).
\end{align}
We estimate the sums on the right-hand side above separately. For the first sum, in the case $X = \mathbb{Q}$, we use \lmaref{gsaidak}. For the third sum, we use part 4 of \lmaref{boundnm}, and \lmaref{sumplogp} to obtain
\begin{align}\label{part4Omega}
    & \sum_{\substack{\mfp,\mfq \\ N(\mfp) N(\mfq) \leq x}} \left( \frac{N(\mfp)^h -h N(\mfp)^2 + h N(\mfp) - 1}{N(\mfp)(N(\mfp)-1)(N(\mfp)^h - 1)} \right) \left( \frac{N(\mfq)^h -h N(\mfq)^2 + h N(\mfq) - 1}{N(\mfq)(N(\mfq)-1)(N(\mfq)^h - 1)} \right) \notag \\
    & = \sum_{\substack{N(\mfp) \leq x/2}} \left( \frac{N(\mfp)^h -h N(\mfp)^2 + h N(\mfp) - 1}{N(\mfp)(N(\mfp)-1)(N(\mfp)^h - 1)} \right) \sum_{\substack{N(\mfq) \leq x/N(\mfp)}} \left( \frac{N(\mfq)^h -h N(\mfq)^2 + h N(\mfq) - 1}{N(\mfq)(N(\mfq)-1)(N(\mfq)^h - 1)} \right) \notag \\
    % & = \sum_{\substack{N(\mfp) \leq x/2}} \left( \frac{N(\mfp)^h -h N(\mfp)^2 + h N(\mfp) - 1}{N(\mfp)(N(\mfp)-1)(N(\mfp)^h - 1)} \right) \Bigg( \sum_{\mfp} \left( \frac{N(\mfp)^h -h N(\mfp)^2 + h N(\mfp) - 1}{N(\mfp)(N(\mfp)-1)(N(\mfp)^h - 1)} \right) \notag \\
    % & \hspace{.5cm} + O \left( \frac{1}{(x/N(\mfp)) \log(x/N(\mfp))}\right)  \Bigg) \notag \\
    & = \left( \sum_{\mfp} \frac{N(\mfp)^h -h N(\mfp)^2 + h N(\mfp) - 1}{N(\mfp)(N(\mfp)-1)(N(\mfp)^h - 1)} \right)^2 + O \left( \frac{\log \log x}{x \log x}\right) .
\end{align}

Now, for the second sum in \eqref{part3Omega}, we use parts 3 and 4 of \lmaref{boundnm}, $\log \log (x/2) = \log \log x + O(1/\log x)$ and \lmaref{gpnt} to obtain
\begin{align}\label{part5Omega}
    &  \sum_{\substack{\mfp,\mfq \\ N(\mfp) N(\mfq) \leq x}} \frac{1}{N(\mfp)} \left( \frac{N(\mfq)^h -h N(\mfq)^2 + h N(\mfq) - 1}{N(\mfq)(N(\mfq)-1)(N(\mfq)^h - 1)} \right)  \notag\\
    & = \sum_{\substack{N(\mfp) \leq x/2}} \frac{1}{N(\mfp)} \sum_{\substack{N(\mfq) \leq x/N(\mfp)}} \left( \frac{N(\mfq)^h -h N(\mfq)^2 + h N(\mfq) - 1}{N(\mfq)(N(\mfq)-1)(N(\mfq)^h - 1)} \right)  \notag\\
    & = \left( \sum_{\mfp} \frac{N(\mfp)^h -h N(\mfp)^2 + h N(\mfp) - 1}{N(\mfp)(N(\mfp)-1)(N(\mfp)^h - 1)}  \right) \left( \log \log (x/2) + \mfa + O \left( \frac{1}{\log x} \right) \right) \notag\\
    & \hspace{.5cm} + O \left( \frac{1}{x} \sum_{\substack{ N(\mfp) \leq x/2}} \frac{1}{\log(x/N(\mfp))} \right) \notag\\
    & = \left( \sum_{\mfp} \frac{N(\mfp)^h -h N(\mfp)^2 + h N(\mfp) - 1}{N(\mfp)(N(\mfp)-1)(N(\mfp)^h - 1)}  \right) \left( \log \log x + \mfa \right) + O \left( \frac{1}{\log x} \right).
\end{align}
For $X = \{ q^z : z \in \mathbb{Z} \}$, similar results to \eqref{part4Omega} and \eqref{part5Omega} can be proved using $x/q$ instead of $x/2$ and the identity $\log \log (x/q) = \log \log x + O(1/\log x)$.

% Using part 3 of \lmaref{boundnm} to the above error term, we obtain
% \begin{align*}
%     &  \sum_{\substack{\mfp,\mfq \\ N(\mfp) N(\mfq) \leq x}} \frac{1}{N(\mfp)} \left( \frac{N(\mfq)^h -h N(\mfq)^2 + h N(\mfq) - 1}{N(\mfq)(N(\mfq)-1)(N(\mfq)^h - 1)} \right) \\
%     & = \left( \sum_{\mfp} \frac{N(\mfp)^h -h N(\mfp)^2 + h N(\mfp) - 1}{N(\mfp)(N(\mfp)-1)(N(\mfp)^h - 1)}  \right) \left( \log \log (x/2) + \mfa + O \left( \frac{1}{\log x} \right) \right) \\
%     & \hspace{.5cm} + O \left( \frac{1}{x} \sum_{\substack{ N(\mfp) \leq x/2}} \frac{1}{\log(x/N(\mfp))} \right).
% \end{align*}
Combining \eqref{part3Omega}, \eqref{part4Omega}, and \eqref{part5Omega} with \lmaref{gsaidak} and \eqref{def-C3}, we obtain
\begin{align*}
    & \sum_{\substack{\mfp,\mfq \\ N(\mfp) N(\mfq) \leq x}} \left( \frac{N(\mfp)^{h} -h N(\mfp) + h - 1}{(N(\mfp)-1)(N(\mfp)^h - 1)} \right) \left( \frac{N(\mfq)^{h} -h N(\mfq) + h - 1}{(N(\mfq)-1)(N(\mfq)^h - 1)} \right) \\
    % & = (\log \log x)^2 + 2 \mfa \log \log x + \mfa^2 + \mathfrak{B} \\
    % & \hspace{.5cm} + 2 \left( \sum_{\mfp} \frac{N(\mfp)^h -h N(\mfp)^2 + h N(\mfp) - 1}{N(\mfp)(N(\mfp)-1)(N(\mfp)^h - 1)} \right) \left( \log \log x + \mfa \right) \\
    % & \hspace{.5cm} + \left( \sum_{\mfp} \frac{N(\mfp)^h -h N(\mfp)^2 + h N(\mfp) - 1}{N(\mfp)(N(\mfp)-1)(N(\mfp)^h - 1)} \right)^2  + O \left( \frac{\log \log x}{\log x} \right) \\
    & = (\log \log x)^2 + 2 \mathfrak{C}_3 \log \log x + \mathfrak{C}_3^2 + \mathfrak{B} + O \left( \frac{\log \log x}{\log x} \right),
\end{align*}
where $\mathfrak{C}_3$ is defined in \eqref{def-C3}.

Combining \eqref{mainpartOmega}, \eqref{bound_sum1}, \eqref{mainpart2Omega}, \eqref{mainpart3Omega}, \eqref{partsOmega}, \eqref{part1Omega}, and \eqref{part2Omega} with the above equation, we obtain the required second moment for $\Omega(\mfm)$ over $h$-free elements.
% as
% \begin{align*}
% \sum_{\mfm \in \mathcal{S}_h(x)}  \Omega^2(n) 
% & = \frac{\kappa}{\zeta_\mcm(h)}  x (\log \log x)^2 + \frac{\kappa (2 \mathfrak{C}_3 + 1)}{\zeta_\mcm(h)} x \log \log x + \frac{\kappa \mathfrak{C}_4}{\zeta_\mcm(h)} x + O_h \left( \frac{x \log \log x}{\log x}\right),
% \end{align*}
% where $\mathfrak{C}_4$ defined in \eqref{def-C4}. 
This completes the proof of \thmref{hfreeOmega}.
\end{proof}
\section{The moments of \texorpdfstring{$\Omega(n)$}{} over \texorpdfstring{$h$}{}-full numbers}

This section proves \thmref{hfullOmega} by establishing the first and the second moments of $\Omega(\mfm)$ over $h$-full elements.

First, we recall the constants
$$\mathfrak{D}_3 = h (\mfa - \log h) + \sum_\mfp \frac{(h+1)N(\mfp) - h N(\mfp)^{1-(1/h)} - h N(\mfp)^{1/h} + h}{N(\mfp) (N(\mfp)^{1/h} - 1) (N(\mfp) - N(\mfp)^{1-(1/h)} +1)},$$
\begin{small}
\begin{align*}
    & \mathfrak{D}_3' \notag \\
    & = h^2(\mfa - \log h) +  \notag \\
    & \sum_\mfp \left( \frac{(2h^2 + 2h - 1)N(\mfp)^{\frac{1 + h}{h}} - (1 + h)^2N(\mfp)^{\frac{2 + h}{h}} - h^2(N(\mfp) - N(\mfp)^{\frac{1}{h}} + 2N(\mfp)^{\frac{2}{h}} - N(\mfp)^{\frac{3}{h}})}{N(\mfp)(-N(\mfp)^{\frac{1 + h}{h}} - N(\mfp)^{\frac{1}{h}} + N(\mfp))(-1 + N(\mfp)^{\frac{1}{h}})^2} \right) ,
\end{align*}
\end{small}
and
$$\mathfrak{D}_4 := \mathfrak{D}_3^2+ \mathfrak{D}_3' + h^2 \mathfrak{B} - \sum_\mfp \left( \frac{h(N(\mfp)^{\frac{1}{h}}-1)+1}{(N(\mfp)^{\frac{1}{h}}-1) (N(\mfp) - N(\mfp)^{1-\frac{1}{h}} +1)} \right)^2.$$
% \begin{thm}
% Let $x > 2$ be a real number. Let $h \geq 2$ be an integer. Let $\mathcal{N}_h(x)$ be the set of $h$-full numbers less than or equal to $x$. Then, we have
% \begin{align*}
% \sum_{\substack{n \leq x \\ n \in \mathcal{N}_h}} \Omega(n) & =   h \gamma_{0,h} x^{1/h} \log \log x + B_3 \gamma_{0,h} x^{1/h} + O_h \left( \frac{x^{1/h}}{\log x} \right),
% \end{align*}
% and
% \begin{align*}
% \sum_{\mfm \in \mathcal{N}_h(x)}  \Omega^2(n) 
% & =  h^2 \gamma_{0,h} x^{1/h} (\log \log x)^2 + (2 B_3 + 1) h \gamma_{0,h} x^{1/h} \log \log x + B_4 \gamma_{0,h} x^{1/h} \\
% & \hspace{.5cm} + O_h \left( \frac{x^{1/h} \log \log x}{\log x}\right).
% \end{align*}
\begin{proof}[\textbf{Proof of \thmref{hfullOmega}}]
Using $\mfm = k \mfp + \mathfrak{y}$ with $n_\mfp(\mathfrak{y}) = 0$ when $n_\mfp(\mfm) = k$,  and \lmaref{hfullidealsrestrict} for a single prime $\mfp$, we obtain
\begin{align}\label{mainomegahfull1Omega}
\sum_{\mfm \in \mathcal{N}_h(x)} \Omega(\mfm) 
% = \sum_{\mfm \in \mathcal{N}_h(x)} \sum_{\substack{\mfp \\ n_\mfp(\mfm) = k}} k 
% = \sum_{N(\mfp) \leq x^{1/h}} \sum_{k = h}^{\lfloor \frac{\log x}{\log N(\mfp)} \rfloor} \sum_{\substack{\mfm \in \mathcal{N}_h(x) \\ n_\mfp(\mfm) = k}} k 
& = \sum_{N(\mfp) \leq x^{1/h}} \sum_{k = h}^{\lfloor \frac{\log x}{\log N(\mfp)} \rfloor} k \sum_{\substack{\mathfrak{y} \in \mathcal{N}_h(x/N(\mfp)^k) \\ n_\mfp(\mathfrak{y}) = 0}} 1 \notag \\
& = \kappa \gamma_{h} x^{1/h} \sum_{N(\mfp) \leq x^{1/h}} \sum_{k = h}^{\lfloor \frac{\log x}{\log N(\mfp)} \rfloor} \left( \frac{k}{N(\mfp)^{k/h} \left( 1 + \frac{N(\mfp)^{-1}}{1-N(\mfp)^{-1/h}}\right)} \right) \notag \\
& \hspace{.2cm} + O_h \left( \sum_{N(\mfp) \leq x^{1/h}} \sum_{k = h}^{\lfloor \frac{\log x}{\log N(\mfp)} \rfloor} k R_{\mathcal{N}_h}(x/N(\mfp)^k) \right).
\end{align}
% By \lmaref{hfullidealsrestrict} for a single prime $\mfp$, we obtain
% \begin{align}
% & \sum_{\substack{\mfm \in \mathcal{N}_h(x)}} \Omega(\mfm) \notag\\
% % & = \sum_{N(\mfp) \leq x^{1/h}} \sum_{k = h}^{\lfloor \frac{\log x}{\log N(\mfp)} \rfloor} k A_{N(\mfp),h}(x/N(\mfp)^k) \notag \\
% & = \gamma_{h} x^{1/h} \sum_{N(\mfp) \leq x^{1/h}} \sum_{k = h}^{\lfloor \frac{\log x}{\log N(\mfp)} \rfloor} \left( \frac{k}{N(\mfp)^{k/h} \left( 1 + \frac{N(\mfp)^{-1}}{1-N(\mfp)^{-1/h}}\right)} \right) \notag \\
% & \hspace{.2cm} + O_h \left( \sum_{N(\mfp) \leq x^{1/h}} \sum_{k = h}^{\lfloor \frac{\log x}{\log N(\mfp)} \rfloor} k R_{\mathcal{N}_h}(x/N(\mfp)^k) \right).
% %\left( \frac{\gamma_{0,h}}{\left( 1 + \frac{N(\mfp)^{-1}}{1-N(\mfp)^{-1/h}}\right)} (x/N(\mfp)^k)^{1/h} + O \left( (x/N(\mfp)^k)^{1/(h+1)}\right) \right).
% \end{align}
First, we study the main term above. Note that
\begin{align}\label{sumkp^k/h}
    \sum_{k = h}^{r} k a^k 
    % & = \frac{a^{r + 1}(ar - r - 1) - a^h((h - 1)a - h)}{(1 - a)^2} \notag \\
    & = \frac{h a^h}{(1-a)} + \frac{a^{h+1}}{(1-a)^2} +  \frac{a^{r + 1}(ar - r - 1)}{(1 - a)^2}.
\end{align}
% \begin{align}\label{sumkp^k/h}
%     & \sum_{k = h}^{r} \frac{k}{N(\mfp)^{k/h}} \notag \\
%     & = \frac{1}{N(\mfp)} \left( \sum_{k = 0}^{r - h} \frac{k}{N(\mfp)^{k/h}} + h \sum_{k = 0}^{r - h} \frac{1}{N(\mfp)^{k/h}} \right) \notag \\
%     & = \frac{\left( r - h + 1 \right)(N(\mfp)^{-(r-h+2)/h} - N(\mfp)^{-(r-h+1)/h}) - N(\mfp)^{-(r-h+2)/h} + N(\mfp)^{-1/h}}{N(\mfp)(N(\mfp)^{-1/h}-1)^2} + \frac{h(N(\mfp)^{-(r -h + 1)/h} - 1)}{N(\mfp)(N(\mfp)^{-1/h} - 1)} \notag \\
%     & = \frac{h}{N(\mfp)(1- N(\mfp)^{-1/h})} + \frac{N(\mfp)^{-1/h}}{N(\mfp)(1- N(\mfp)^{-1/h})^2} \notag \\
%     & \hspace{.5cm} + \frac{\left( r - h + 1 \right)(N(\mfp)^{-(r-h+2)/h} - N(\mfp)^{-(r-h+1)/h}) - N(\mfp)^{-(r-h+2)/h} - h (1- N(\mfp)^{-1/h}) N(\mfp)^{-(r -h + 1)/h}}{N(\mfp)(1- N(\mfp)^{-1/h})^2}.
% \end{align}
Using the above formula with $r = \lfloor \frac{\log x}{\log N(\mfp)} \rfloor$ and $a = N(\mfp)^{-1/h}$, we obtain
\begin{align}\label{need-es}
    & \sum_{N(\mfp) \leq x^{1/h}} \sum_{k = h}^{\lfloor \frac{\log x}{\log N(\mfp)} \rfloor} \left( \frac{k}{N(\mfp)^{k/h} \left( 1 + \frac{N(\mfp)^{-1}}{1-N(\mfp)^{-1/h}}\right)} \right) \notag\\
    & = \sum_{N(\mfp) \leq x^{1/h}} \Bigg( \frac{h}{N(\mfp)(1- N(\mfp)^{-1/h}+ N(\mfp)^{-1})} \notag \\
    & \hspace{.5cm} + \frac{N(\mfp)^{-1/h}}{N(\mfp)(1- N(\mfp)^{-1/h}) (1- N(\mfp)^{-1/h}+ N(\mfp)^{-1})} \Bigg) \notag \\
    & \hspace{.5cm} + O_h \left( \sum_{N(\mfp) \leq x^{1/h}} \frac{\left( \lfloor \frac{\log x}{\log N(\mfp)} \rfloor \right) N(\mfp)^{-(\lfloor \frac{\log x}{\log N(\mfp)} \rfloor+1)/h}}{(1- N(\mfp)^{-1/h}) (1- N(\mfp)^{-1/h}+ N(\mfp)^{-1})} \right).
\end{align}
By partial summation and \lmaref{gpnt}, it follows that for any natural number $k$,
\begin{equation}\label{sum1/logp}
\sum_{N(\mfp) \leq y} \frac{1}{(\log N(\mfp))^k} \ll \frac{y}{(\log y)^{k+1}}.   
\end{equation}
Using \eqref{sum1/logp} with $k=1$ and $y = x^{1/h}$, and $x -1 < \lfloor x \rfloor \leq x$, we can estimate the error term in \eqref{need-es} as
$$\sum_{N(\mfp) \leq x^{1/h}} \frac{\left( \lfloor \frac{\log x}{\log N(\mfp)} \rfloor \right) N(\mfp)^{-(\lfloor \frac{\log x}{\log N(\mfp)} \rfloor+1)/h}}{(1- N(\mfp)^{-1/h})  (1- N(\mfp)^{-1/h}+ N(\mfp)^{-1})} \ll_h \frac{\log x}{x^{1/h}} \sum_{N(\mfp) \leq x^{1/h}} \frac{1}{\log N(\mfp)} \ll_h \frac{1}{\log x}.$$
Thus, we obtain
\begin{align}\label{sum_needOmega}
& \kappa \gamma_{h} x^{1/h} \sum_{N(\mfp) \leq x^{1/h}} \sum_{k = h}^{\lfloor \frac{\log x}{\log N(\mfp)} \rfloor} \left( \frac{k}{N(\mfp)^{k/h} \left( 1 + \frac{N(\mfp)^{-1}}{1-N(\mfp)^{-1/h}}\right)} \right) \notag \\
% & = \gamma_{h} x^{1/h} \sum_{N(\mfp) \leq x^{1/h}} \Bigg( \frac{h}{N(\mfp)(1- N(\mfp)^{-1/h}+ N(\mfp)^{-1})} \notag \\ 
% & \hspace{.5cm} + \frac{N(\mfp)^{-1/h}}{N(\mfp)(1- N(\mfp)^{-1/h}) (1- N(\mfp)^{-1/h}+ N(\mfp)^{-1})} \Bigg) + O_h \left( \frac{x^{1/h}}{\log x} \right) \notag\\
& = h \kappa \gamma_{h} x^{1/h} \sum_{N(\mfp) \leq x^{1/h}}  \frac{1}{N(\mfp)(1- N(\mfp)^{-1/h}+ N(\mfp)^{-1})} \notag \\
& \hspace{.5cm} + \kappa \gamma_{h} x^{1/h} \sum_{N(\mfp) \leq x^{1/h}} \frac{1}{(N(\mfp)^{1/h}- 1) (N(\mfp)- N(\mfp)^{1-1/h}+ 1)} + O_h \left( \frac{x^{1/h}}{\log x} \right).
% & = \gamma_{0,h} x^{1/h} \sum_{N(\mfp) \leq x^{1/h}} \frac{1}{N(\mfp)\left( 1- N(\mfp)^{-1/h}+N(\mfp)^{-1} \right)} - \gamma_{0,h} x^{1/h} \sum_{N(\mfp) \leq x^{1/h}} \frac{(N(\mfp)^{-1/h})^{\lfloor \frac{\log x}{\log N(\mfp)} \rfloor - h +1}}{N(\mfp)\left( 1- N(\mfp)^{-1/h}+N(\mfp)^{-1} \right)}.
\end{align}

% Using $\lfloor x \rfloor \geq x-1$, we bound the second term above with
% \begin{align}\label{req1Omega}
%     \gamma_{0,h} x^{1/h} \sum_{N(\mfp) \leq x^{1/h}} \frac{(N(\mfp)^{-1/h})^{\lfloor \frac{\log x}{\log N(\mfp)} \rfloor - h +1}}{N(\mfp)\left( 1- N(\mfp)^{-1/h}+N(\mfp)^{-1} \right)} & \ll_h  x^{1/h} \sum_{N(\mfp) \leq x^{1/h}} \frac{x^{-1/h} N(\mfp)}{N(\mfp)\left( 1- N(\mfp)^{-1/h}+N(\mfp)^{-1} \right)} \notag\\
%     & \ll_h  \sum_{N(\mfp) \leq x^{1/h}} \frac{1}{\left( 1- N(\mfp)^{-1/h}+N(\mfp)^{-1} \right)}.
% \end{align} 
Using part 4 of \lmaref{boundnm} with $\alpha = 1+ (1/h)$, we can estimate the second sum on the right side above as
\begin{align*}
    & \sum_{N(\mfp) \leq x^{1/h}} \frac{1}{(N(\mfp)^{1/h}- 1) (N(\mfp)- N(\mfp)^{1-1/h}+ 1)} \\
    & = \sum_{\mfp} \frac{1}{(N(\mfp)^{1/h}- 1) (N(\mfp)- N(\mfp)^{1-1/h}+ 1)} + O_h \left( \frac{1}{x^{1/h^2} (\log x)} \right).
\end{align*}
Moreover, for the first sum on the right side in \eqref{sum_needOmega}, using parts 3 and 4 of \lmaref{boundnm} and the following relation
$$ \frac{1}{N(\mfp)\left( 1- N(\mfp)^{-1/h}+N(\mfp)^{-1} \right)} = \frac{1}{N(\mfp)} + \frac{N(\mfp)^{-1/h} - N(\mfp)^{-1}}{ N(\mfp) \left( 1- N(\mfp)^{-1/h}+N(\mfp)^{-1} \right) },$$
we obtain
\begin{align*}
    & \sum_{N(\mfp) \leq x^{1/h}} \frac{1}{N(\mfp)\left( 1- N(\mfp)^{-1/h}+N(\mfp)^{-1} \right)} \notag \\
    & = \log \log x + \mfa - \log h + \sum_{\mfp}  \frac{N(\mfp)^{-1/h} - N(\mfp)^{-1}}{ N(\mfp) \left( 1- N(\mfp)^{-1/h}+N(\mfp)^{-1} \right) } + O_h \left( \frac{1}{\log x} \right).
\end{align*}
Also, a slight simplification yields
\begin{align}\label{simplification}
    &  h \sum_{\mfp}  \frac{N(\mfp)^{-1/h} - N(\mfp)^{-1}}{ N(\mfp) \left( 1- N(\mfp)^{-1/h}+N(\mfp)^{-1} \right) } + \sum_{\mfp} \frac{1}{(N(\mfp)^{1/h}- 1) (N(\mfp)- N(\mfp)^{1-1/h}+ 1)} \notag \\
    & = \sum_\mfp \frac{h(N(\mfp)-N(\mfp)^{1-(1/h)}-N(\mfp)^{1/h} + 1)+N(\mfp)}{N(\mfp) (N(\mfp)^{1/h} - 1) (N(\mfp) - N(\mfp)^{1-(1/h)} +1)}. 
\end{align}
Combining the last four results with \eqref{mainomegahfull1Omega}, we obtain
\begin{align}\label{hfullfinalOmega}
\sum_{\substack{\mfm \in \mathcal{N}_h(x)}} \Omega(\mfm) & =  h \kappa \gamma_{h} x^{1/h} \log \log x + \mathfrak{D}_3 \kappa \gamma_{h} x^{1/h} + O_h \left( \frac{x^{1/h}}{\log x} \right) \notag \\
& \hspace{.5cm} + O_h \left( \sum_{N(\mfp) \leq x^{1/h}} \sum_{k = h}^{\lfloor \frac{\log x}{\log N(\mfp)} \rfloor} k R_{\mathcal{N}_h}(x/N(\mfp)^k)\right),
\end{align}
where $\mathfrak{D}_3$ is defined in \eqref{def-D3}.

From \rmkref{remark2}, recall that, we can write $R_{\mathcal{N}_h}(x) \ll x^{\upsilon/h}$ where $h/(h+1) \leq \upsilon < 1$. Using this in the error term above, using \eqref{sumkp^k/h} with $r = \lfloor \frac{\log x}{\log N(\mfp)} \rfloor$ and $a = N(\mfp)^{-\upsilon/h}$, and part 1 of \lmaref{boundnm} with $\alpha = \upsilon$, we obtain
\begin{align*}
    \sum_{N(\mfp) \leq x^{1/h}} \sum_{k = h}^{\left\lfloor \frac{\log x}{\log N(\mfp)}  \right\rfloor} k R_{\mathcal{N}_h}(x/N(\mfp)^k)  & \ll_h x^{\upsilon/h} \sum_{N(\mfp) \leq x^{1/h}} \frac{1}{N(\mfp)^\upsilon} \ll_h \frac{x^{1/h}}{\log x}.
\end{align*}
Finally, inserting the above back into \eqref{hfullfinalOmega}, we obtain
\begin{align*}
\sum_{\substack{\mfm \in \mathcal{N}_h(x)}} \Omega(\mfm) & =   h \kappa \gamma_{h} x^{1/h} \log \log x + \mathfrak{D}_3 \kappa \gamma_{h} x^{1/h} + O_h \left( \frac{x^{1/h}}{\log x} \right).
\end{align*}
This completes the first part of the proof. 

For the second part, note that
\begin{align}\label{hfullp1Omega}
    \sum_{\substack{\mfm \in \mathcal{N}_h(x)}} \Omega^2(\mfm) & = \sum_{\substack{\mfm \in \mathcal{N}_h(x)}} \left( \sum_{\substack{\mfp \\ n_\mfp(\mfm) = k}} k \right)^2 \notag \\
    & = \sum_{\substack{\mfm \in \mathcal{N}_h(x)}} \sum_{\substack{\mfp \\ n_\mfp(\mfm) = k}} k^2 + \sum_{\substack{\mfm \in \mathcal{N}_h(x)}} \sum_{\substack{\mfp,\mfq \\ n_\mfp(\mfm) = k, n_\mfq(\mfm) = l,  \mfp \neq \mfq }} (k l).
\end{align}
For the first sum on the right side above, we rewrite the sum as
$$\sum_{\substack{\mfm \in \mathcal{N}_h(x)}} \sum_{\substack{\mfp \\ n_\mfp(\mfm) = k}} k^2 = \sum_{N(\mfp) \leq x^{1/h}} \sum_{k = h}^{\lfloor \frac{\log x}{\log N(\mfp)} \rfloor} \sum_{\substack{\mfm \in \mathcal{N}_h(x) \\ n_\mfp(\mfm) = k}} k^2 = \sum_{N(\mfp) \leq x^{1/h}} \sum_{k = h}^{\lfloor \frac{\log x}{\log N(\mfp)} \rfloor} k^2 \sum_{\substack{\mathfrak{y} \in \mathcal{N}_h(x/N(\mfp)^k) \\ n_\mfp(\mathfrak{y}) = 0}} 1.$$
Working similarly to \eqref{mainomegahfull1Omega}, we obtain
\begin{align}\label{need-for-constant}
    \sum_{\substack{\mfm \in \mathcal{N}_h(x)}} \sum_{\substack{\mfp \\ n_\mfp(\mfm) = k}} k^2 & = \kappa \gamma_{h} x^{1/h} \sum_{N(\mfp) \leq x^{1/h}} \frac{1 - N(\mfp)^{-1/h}}{\left( 1 - N(\mfp)^{-1/h} + N(\mfp)^{-1} \right)} \sum_{k = h}^{\lfloor \frac{\log x}{\log N(\mfp)} \rfloor} \left( \frac{k^2}{N(\mfp)^{k/h}} \right) \notag \\
    & \hspace{.5cm} + O_h \left( \sum_{N(\mfp) \leq x^{1/h}} \sum_{k = h}^{\lfloor \frac{\log x}{\log N(\mfp)} \rfloor} k^2 R_{\mathcal{N}_h}(x/N(\mfp)^k) \right).
\end{align}
Note that
\begin{align}\label{need-k^2}
    & \sum_{k = h}^r k^2 a^k \notag \\
    & = \frac{(a^2r^2 + (-2r^2 - 2r + 1)a + (r + 1)^2)a^{r + 1} - a^h((h - 1)^2a^2 + (-2h^2 + 2h + 1)a + h^2)}{(a - 1)^3} \notag \\
    & = \frac{h^2 a^h}{(1-a)^3} + \frac{a^{h+1}(-2h^2+2h+1)}{(1-a)^3} + \frac{a^{h+2}(h-1)^2}{(1-a)^3} \notag \\
    & \hspace{.5cm} + \frac{a^{r + 1}(a^2r^2 + (-2r^2 - 2r + 1)a + (r + 1)^2)}{(a - 1)^3}.
\end{align}
Using the above relation with $r = \lfloor \frac{\log x}{\log N(\mfp)} \rfloor$ and $a = N(\mfp)^{-1/h}$, using $x -1 < \lfloor x \rfloor \leq x$, and using \eqref{sum1/logp} with $k=2$ and $y = x^{1/h}$, we obtain
\begin{small}
\begin{align*}
    & \sum_{N(\mfp) \leq x^{1/h}} \frac{1 - N(\mfp)^{-1/h}}{\left( 1 - N(\mfp)^{-1/h} + N(\mfp)^{-1} \right)} \sum_{k = h}^{\lfloor \frac{\log x}{\log N(\mfp)} \rfloor} \left( \frac{k^2}{N(\mfp)^{k/h}} \right) \\
    & = \sum_{N(\mfp) \leq x^{1/h}} \frac{h^2 N(\mfp)^{-1} + N(\mfp)^{-1-\frac{1}{h}}(-2h^2+2h+1) + N(\mfp)^{-1-\frac{2}{h}}(h-1)^2}{(1-N(\mfp)^{-\frac{1}{h}})^2 (1 - N(\mfp)^{-\frac{1}{h}} + N(\mfp)^{-1})} \\
    & \hspace{.5cm} + O \left( \sum_{N(\mfp) \leq x^{1/h}} \frac{\left( \frac{\log x}{\log N(\mfp)} \right)^2}{N(\mfp)^{\frac{ \frac{\log x}{\log N(\mfp)} }{h}}}\right) \\
    & = \sum_{N(\mfp) \leq x^{1/h}}\frac{h^2 + N(\mfp)^{-\frac{1}{h}}(-2h^2+2h+1) + N(\mfp)^{-\frac{2}{h}}(h-1)^2}{N(\mfp) (1-N(\mfp)^{-\frac{1}{h}})^2 (1 - N(\mfp)^{-\frac{1}{h}} + N(\mfp)^{-1})} \\
    & \hspace{.5cm} + O \left( \frac{(\log x)^2}{x^{1/h}} \sum_{N(\mfp) \leq x^{1/h}} \frac{1}{(\log N(\mfp))^2} \right) \\
    & = \sum_{N(\mfp) \leq x^{1/h}} \frac{h^2}{N(\mfp)} +  \\
    & \sum_{N(\mfp) \leq x^{1/h}} \left( \frac{(2h^2 + 2h - 1)N(\mfp)^{\frac{1 + h}{h}} - (1 + h)^2N(\mfp)^{\frac{2 + h}{h}} - h^2(N(\mfp) - N(\mfp)^{\frac{1}{h}} + 2N(\mfp)^{\frac{2}{h}} - N(\mfp)^{\frac{3}{h}})}{N(\mfp)(-N(\mfp)^{\frac{1 + h}{h}} - N(\mfp)^{\frac{1}{h}} + N(\mfp))(-1 + N(\mfp)^{\frac{1}{h}})^2} \right) \\
    & + O_h \left( \frac{1}{\log x} \right).
\end{align*}
\end{small}
Notice that the second sum on the right side above is $O(1)$. Thus, using parts 3 and 4 of \lmaref{boundnm} with $\alpha = 1+(1/h)$, we obtain
\begin{small}
\begin{align*}
    & \sum_{N(\mfp) \leq x^{1/h}} \frac{1 - N(\mfp)^{-1/h}}{\left( 1 - N(\mfp)^{-1/h} + N(\mfp)^{-1} \right)} \sum_{k = h}^{\lfloor \frac{\log x}{\log N(\mfp)} \rfloor} \left( \frac{k^2}{N(\mfp)^{k/h}} \right) \\
    % & = h^2 \left(\log \log x - \log h + \mfa + O_h \left( \frac{1}{\log x}  \right) \right) \\
    % & + \sum_\mfp \left( \frac{(2h^2 + 2h - 1)N(\mfp)^{\frac{1 + h}{h}} - (1 + h)^2N(\mfp)^{\frac{2 + h}{h}} - h^2(N(\mfp) - N(\mfp)^{\frac{1}{h}} + 2N(\mfp)^{\frac{2}{h}} - N(\mfp)^{\frac{3}{h}})}{N(\mfp)(-N(\mfp)^{\frac{1 + h}{h}} - N(\mfp)^{\frac{1}{h}} + N(\mfp))(-1 + N(\mfp)^{\frac{1}{h}})^2} \right) \\
    % & + O_h \left( \sum_{N(\mfp) > x^{1/h}} \frac{1}{N(\mfp)^{1+\frac{1}{h}}} \right) + O_h \left( \frac{1}{\log x} \right) \\
    & = h^2 \log \log x + h^2(\mfa - \log h) \\
    & + \sum_\mfp \left( \frac{(2h^2 + 2h - 1)N(\mfp)^{\frac{1 + h}{h}} - (1 + h)^2N(\mfp)^{\frac{2 + h}{h}} - h^2(N(\mfp) - N(\mfp)^{\frac{1}{h}} + 2N(\mfp)^{\frac{2}{h}} - N(\mfp)^{\frac{3}{h}})}{N(\mfp)(-N(\mfp)^{\frac{1 + h}{h}} - N(\mfp)^{\frac{1}{h}} + N(\mfp))(-1 + N(\mfp)^{\frac{1}{h}})^2} \right) \\
    & + O_h \left( \frac{1}{\log x} \right).
\end{align*}
\end{small}
Using $R_{\mathcal{N}_h}(x) \ll x^{\upsilon/h}$ where $h/(h+1) \leq \upsilon < 1$, \eqref{need-k^2} with $r = \lfloor \frac{\log x}{\log N(\mfp)} \rfloor$ and $a = N(\mfp)^{-\upsilon/h}$, and part 1 of \lmaref{boundnm} with $\alpha = \upsilon$, we obtain
\begin{align*}
    & \sum_{N(\mfp) \leq x^{1/h}} \sum_{k = h}^{\lfloor \frac{\log x}{\log N(\mfp)} \rfloor} k^2 R_{\mathcal{N}_h}(x/N(\mfp)^k)  \ll_h x^{\upsilon/h} \sum_{N(\mfp) \leq x^{1/h}} \frac{1}{N(\mfp)^\upsilon} \ll_h \frac{x^{1/h}}{\log x}.
\end{align*}
Inserting the last two results into \eqref{need-for-constant}, we obtain
\begin{small}
\begin{align}\label{need_later-D4}
& \sum_{\substack{\mfm \in \mathcal{N}_h(x)}} \sum_{\substack{\mfp \\ n_\mfp(\mfm) = k}} k^2 \notag \\
& = \kappa \gamma_{h} x^{1/h} \Bigg( h^2 \log \log x + h^2(\mfa - \log h) \notag \\
& + \sum_\mfp \left( \frac{(2h^2 + 2h - 1)N(\mfp)^{\frac{1 + h}{h}} - (1 + h)^2N(\mfp)^{\frac{2 + h}{h}} - h^2(N(\mfp) - N(\mfp)^{\frac{1}{h}} + 2N(\mfp)^{\frac{2}{h}} - N(\mfp)^{\frac{3}{h}})}{N(\mfp)(-N(\mfp)^{\frac{1 + h}{h}} - N(\mfp)^{\frac{1}{h}} + N(\mfp))(-1 + N(\mfp)^{\frac{1}{h}})^2} \right) \Bigg) \notag \\
& + O_h \left( \frac{x^{1/h}}{\log x} \right) \notag \\
& = h^2 \kappa \gamma_{h} x^{1/h} \log \log x + \mathfrak{D}_3' \kappa \gamma_{h} x^{1/h} + O_h \left( \frac{x^{1/h}}{\log x} \right),
\end{align}
\end{small}
where $\mathfrak{D}_3'$ is defined in \eqref{def-D3'}.

For the second sum in \eqref{hfullp1Omega}, we rewrite the sum, use \lmaref{hfullidealsrestrict} with two distinct primes $\mfp$ and $\mfq$, and use $R_{\mathcal{N}_h}(x) \ll x^{\upsilon/h}$ to obtain
\begin{align}\label{hfullp2Omega}
    & \sum_{\substack{\mfm \in \mathcal{N}_h(x)}} \sum_{\substack{\mfp,\mfq \\ n_\mfp(\mfm) = k, n_\mfq(\mfm) = l,  \mfp \neq \mfq }} (k l) \notag \\
    & = \sum_{\substack{\mfp,\mfq \\ \mfp \neq \mfq, N(\mfp) N(\mfq) \leq x^{1/h}}} \sum_{k=h}^{\lfloor \frac{\log x}{\log N(\mfp)} \rfloor} \sum_{l=h}^{\lfloor \frac{\log x}{\log N(\mfq)} \rfloor} \sum_{\substack{\mfm \in \mathcal{N}_h(x/(N(\mfp)^k N(\mfq)^l)) \\ n_\mfp(\mathfrak{y}) = n_\mfq(\mathfrak{y}) = 0}} (k l) \notag \\
    % & = \sum_{\substack{N(\mfp),\mfq \\ N(\mfp) \neq \mfq, pq \leq x^{1/h}}} \sum_{k=h}^{\lfloor \frac{\log x}{\log N(\mfp)} \rfloor} \sum_{l=h}^{\lfloor \frac{\log x}{\log \mfq} \rfloor} \left( \frac{\gamma_{0,h}}{\left( 1 + \frac{N(\mfp)^{-1}}{1 - N(\mfp)^{-1/h}} \right) \left( 1 + \frac{\mfq^{-1}}{1 - \mfq^{-1/h}} \right)} \left( \frac{x}{N(\mfp)^k \mfq^l} \right)^\frac{1}{h} + O_h \left( \left( \frac{x}{N(\mfp)^k \mfq^l} \right)^\frac{1}{h+1} \right) \right) \notag \\
    & = \kappa \gamma_{h} x^{1/h} \sum_{\substack{\mfp,\mfq \\ \mfp \neq \mfq \\ N(\mfp) N(\mfq) \leq x^{1/h}}} \sum_{k=h}^{\lfloor \frac{\log x}{\log N(\mfp)} \rfloor} \sum_{l=h}^{\lfloor \frac{\log x}{\log N(\mfq)} \rfloor} \frac{k}{N(\mfp)^{k/h} \left( 1 + \frac{N(\mfp)^{-1}}{1 - N(\mfp)^{-1/h}} \right)} \frac{l}{N(\mfq)^{l/h}  \left( 1 + \frac{N(\mfq)^{-1}}{1 - N(\mfq)^{-1/h}} \right)} \notag \\
    & \hspace{.5cm} + O_h \left( x^\frac{\upsilon}{h}   \sum_{\substack{\mfp,\mfq \\ \mfp \neq \mfq \\ N(\mfp) N(\mfq) \leq x^{1/h}}} \left(\sum_{k=h}^{\lfloor \frac{\log x}{\log N(\mfp)} \rfloor} \frac{k}{N(\mfp)^{k\upsilon/h}} \right) \left( \sum_{l=h}^{\lfloor \frac{\log x}{\log N(\mfq)} \rfloor} \frac{l}{N(\mfq)^{l \upsilon/h}} \right) \right).
\end{align}
Note that taking $r \rightarrow \infty$ and $0 < a < 1$ in \eqref{sumkp^k/h}, we have
$$\sum_{k=h}^\infty k a^k \ll_h \frac{a^h}{(1-a)^2}.$$
Thus, when $X = \mathbb{Q}$, using $a = N(\mfp)^{-\upsilon/h}$ and $N(\mfq)^{-\upsilon/h}$, the bounds $N(\mfp)^{-\upsilon/h}  \leq 2^{-\upsilon/h}$ and $N(\mfq)^{-\upsilon/h} \leq 2^{-\upsilon/h}$, part 1 of \lmaref{boundnm} and \lmaref{sumplogp}, we bound the error term in \eqref{hfullp2Omega} with
% \begin{align*}
%  & x^\frac{1}{h+1}   \sum_{\substack{N(\mfp),\mfq \\ N(\mfp) \neq \mfq, pq \leq x^{1/h}}} \left(\sum_{k=h}^{\lfloor \frac{\log x}{\log N(\mfp)} \rfloor} \frac{k}{N(\mfp)^{k/(h+1)}} \right) \left( \sum_{l=h}^{\lfloor \frac{\log x}{\log \mfq} \rfloor} \frac{l}{\mfq^{l/(h+1)}} \right) \\  
%  & \ll x^\frac{1}{h+1}   \sum_{\substack{N(\mfp),\mfq \\ N(\mfp) \neq \mfq, pq \leq x^{1/h}}} \left(\sum_{k=h}^{\infty} \frac{k}{N(\mfp)^{k/(h+1)}} \right) \left( \sum_{l=h}^{\infty} \frac{l}{\mfq^{l/(h+1)}} \right) \\
%  & \ll_h x^{\frac{1}{h+1}} \sum_{\substack{N(\mfp) \\ N(\mfp) \leq x^{1/h}/2}} \left(\sum_{k=h}^{\lfloor \frac{\log x}{\log N(\mfp)} \rfloor} \frac{k}{N(\mfp)^{k/(h+1)}} \right)  \left(\sum_{\substack{\mfq \\ \mfq \leq x^{1/h}/N(\mfp)}} \frac{1}{\mfq^{h/(h+1)}} + \sum_{\substack{\mfq \\ \mfq \leq x^{1/h}/N(\mfp)}} \frac{\left( \lfloor \frac{\log x}{\log \mfq} \rfloor \right) \mfq^{-(\lfloor \frac{\log x}{\log \mfq} \rfloor-h+1)/(h+1)}}{\mfq^{h/(h+1)}(1- \mfq^{-1/(h+1)})^2} \right) \\
%  & \ll_h x^{\frac{1}{h+1}} \sum_{\substack{N(\mfp) \\ N(\mfp) \leq x^{1/h}/2}} \left(\sum_{k=h}^{\lfloor \frac{\log x}{\log N(\mfp)} \rfloor} \frac{k}{N(\mfp)^{k/(h+1)}} \right) \left( \frac{x^{1/(h(h+1))}}{N(\mfp)^{1/(h+1)} \log(x^{1/h}/N(\mfp))} \right)
% \end{align*}
\begin{align}\label{hfullp3Omega}
     % & x^\frac{1}{h+1}   \sum_{\substack{N(\mfp),\mfq \\ N(\mfp) \neq \mfq, pq \leq x^{1/h}}} \left(\sum_{k=h}^{\lfloor \frac{\log x}{\log N(\mfp)} \rfloor} \frac{k}{N(\mfp)^{k/(h+1)}} \right) \left( \sum_{l=h}^{\lfloor \frac{\log x}{\log \mfq} \rfloor} \frac{l}{\mfq^{l/(h+1)}} \right) \notag \\  
     & \ll x^\frac{\upsilon}{h}  \sum_{\substack{\mfp,\mfq \\ \mfp \neq \mfq, N(\mfp) N(\mfq) \leq x^{1/h}}} \left(\sum_{k=h}^{\infty} \frac{k}{N(\mfp)^{k\upsilon/h}} \right) \left( \sum_{l=h}^{\infty} \frac{l}{N(\mfq)^{l \upsilon/h}} \right)\notag \\
     % & \ll_h x^{\frac{1}{h+1}} \sum_{\substack{N(\mfp),\mfq \\ pq \leq x^{1/h}}} \frac{1}{N(\mfp)^{h/(h+1)} \mfq^{h/(h+1)}} \notag \\
     & \ll_h x^\frac{\upsilon}{h} \sum_{\substack{N(\mfp) \leq x^{1/h}/2}} \frac{1}{N(\mfp)^{\upsilon}}  \sum_{\substack{N(\mfq) \leq x^{1/h}/N(\mfp)}} \frac{1}{N(\mfq)^{\upsilon}}  \notag \\
     & \ll_h x^{\frac{1}{h}} \sum_{\substack{N(\mfp) \leq x^{1/h}/2}} \frac{1}{N(\mfp) \log(x^{1/h}/N(\mfp))} \notag\\
     & \ll_h \frac{x^{\frac{1}{h}} \log \log x}{\log x}.
\end{align}
For $X = \{ q^z : z \in \mathbb{Z} \}$, a similar result as above can be proved using $x^{1/h}/q$ instead of $x^{1/h}/2$, and the bounds $N(\mfp)^{-\upsilon/h} \leq q^{-\upsilon/h}$ and $N(\mfq)^{-\upsilon/h} \leq q^{-\upsilon/h}$.

Next, we estimate the main term in \eqref{hfullp2Omega}. First note that, by \eqref{sumkp^k/h}, we have
\begin{align*}
    \sum_{k=h}^{\lfloor \frac{\log x}{\log N(\mfp)} \rfloor} \frac{k}{N(\mfp)^{k/h} \left( 1 + \frac{N(\mfp)^{-1}}{1 - N(\mfp)^{-1/h}} \right)} 
    & = \frac{h(1- N(\mfp)^{-1/h})+N(\mfp)^{-1/h}}{N(\mfp)(1- N(\mfp)^{-1/h}) (1- N(\mfp)^{-1/h}+ N(\mfp)^{-1})} \\
    & \hspace{.5cm} + O_h \left( \frac{\left( \lfloor \frac{\log x}{\log N(\mfp)} \rfloor \right) N(\mfp)^{-(\lfloor \frac{\log x}{\log N(\mfp)} \rfloor+1)/h}}{(1- N(\mfp)^{-1/h}) (1- N(\mfp)^{-1/h}+ N(\mfp)^{-1})} \right).
\end{align*}
% $$\sum_{k=h}^{\lfloor \frac{\log x}{\log N(\mfp)} \rfloor} \frac{1}{N(\mfp)^{k/h} \left( 1 + \frac{N(\mfp)^{-1}}{1 - N(\mfp)^{-1/h}} \right)} = \frac{1}{N(\mfp)(1-N(\mfp)^{-1/h}+N(\mfp)^{-1})} - \frac{N(\mfp)^{-\frac{1}{h} \left( \lfloor \frac{\log x}{\log N(\mfp)} \rfloor - h + 1\right)}}{N(\mfp)(1-N(\mfp)^{-1/h}+N(\mfp)^{-1})}.$$
Thus, using a similar result for a prime $\mfq$ as above and the symmetry of two primes $\mfp$ and $\mfq$, we deduce
\begin{align*}
    & \kappa \gamma_{h} x^{1/h} \sum_{\substack{\mfp,\mfq \\ \mfp \neq \mfq, N(\mfp) N(\mfq) \leq x^{1/h}}} \bigg( \sum_{k=h}^{\lfloor \frac{\log x}{\log N(\mfp)} \rfloor} \sum_{l=h}^{\lfloor \frac{\log x}{\log N(\mfq)} \rfloor} \frac{k}{N(\mfp)^{k/h} \left( 1 + \frac{N(\mfp)^{-1}}{1 - N(\mfp)^{-1/h}} \right)}  \cdot \\
    & \hspace{1cm} \frac{l}{N(\mfq)^{l/h}  \left( 1 + \frac{N(\mfq)^{-1}}{1 - N(\mfq)^{-1/h}} \right)} \bigg) \\
    & = \kappa \gamma_{h} x^{1/h}\sum_{\substack{\mfp,\mfq \\ \mfp \neq \mfq, N(\mfp) N(\mfq) \leq x^{1/h}}}  \bigg( \frac{h(1- N(\mfp)^{-1/h})+N(\mfp)^{-1/h}}{N(\mfp)(1- N(\mfp)^{-1/h}) (1- N(\mfp)^{-1/h}+ N(\mfp)^{-1})}  \cdot \\
    & \hspace{1cm} \frac{h(1- N(\mfq)^{-1/h})+N(\mfq)^{-1/h}}{N(\mfq)(1- N(\mfq)^{-1/h}) (1- N(\mfq)^{-1/h}+ N(\mfq)^{-1})} \bigg) + 2 J_1 + J_2,
\end{align*}
where
\begin{align*}
    J_1 & \ll_h \kappa \gamma_{h} x^{1/h} \sum_{\substack{\mfp,\mfq \\ \mfp \neq \mfq, N(\mfp) N(\mfq) \leq x^{1/h}}} \bigg( \frac{\left( \lfloor \frac{\log x}{\log N(\mfp)} \rfloor \right) N(\mfp)^{-(\lfloor \frac{\log x}{\log N(\mfp)} \rfloor+1)/h}}{(1- N(\mfp)^{-1/h}) (1- N(\mfp)^{-1/h}+ N(\mfp)^{-1})}  \cdot  \\
    & \hspace{1cm}\frac{h(1- N(\mfq)^{-1/h})+N(\mfq)^{-1/h}}{N(\mfq)(1- N(\mfq)^{-1/h}) (1- N(\mfq)^{-1/h}+ N(\mfq)^{-1})} \bigg),
\end{align*}
and
\begin{align*}
    J_2 & \ll_h \kappa \gamma_{h} x^{1/h} \sum_{\substack{\mfp,\mfq \\ \mfp \neq \mfq, N(\mfp) N(\mfq) \leq x^{1/h}}} \bigg( \frac{\left( \lfloor \frac{\log x}{\log N(\mfp)} \rfloor \right) N(\mfp)^{-(\lfloor \frac{\log x}{\log N(\mfp)} \rfloor+1)/h}}{(1- N(\mfp)^{-1/h}) (1- N(\mfp)^{-1/h}+ N(\mfp)^{-1})} \cdot \\
    & \hspace{1cm} \frac{\left( \lfloor \frac{\log x}{\log N(\mfq)} \rfloor \right) N(\mfq)^{-(\lfloor \frac{\log x}{\log N(\mfq)} \rfloor+1)/h}}{(1- N(\mfq)^{-1/h}) (1- N(\mfq)^{-1/h}+ N(\mfq)^{-1})} \bigg).
\end{align*}
For $X = \mathbb{Q}$, using $\lfloor x \rfloor \geq x-1$, part 3 of \lmaref{boundnm}, and \eqref{sum1/logp}, we obtain
\begin{align*}
    J_1 & \ll_h  \sum_{\substack{N(\mfp) \leq x^{1/h}/2}} \frac{\log x}{\log N(\mfp)} \left( \sum_{\substack{N(\mfq) \leq x^{1/h}/N(\mfp)}} \frac{1}{N(\mfq)}\right) \\
    & \ll_h  (\log x) (\log \log x) \sum_{\substack{N(\mfp) \leq x^{1/h}/2}} \frac{1}{\log N(\mfp)} \\
    & \ll_h \frac{x^{1/h} \log \log x}{\log x}.
\end{align*}
Moreover, using \lmaref{sumplogx/N(mfp)^2} and partial summation, we obtain
\begin{align*}
    J_2 & \ll_h \sum_{\substack{N(\mfp) \leq x^{1/h}/2}} \left( \frac{\left( \lfloor \frac{\log x}{\log N(\mfp)} \rfloor \right) N(\mfp)^{-(\lfloor \frac{\log x}{\log N(\mfp)} \rfloor-h+1)/h}}{N(\mfp)(1- N(\mfp)^{-1/h})^2} \right) \left( \sum_{\substack{N(\mfq) \leq x^{1/h}/N(\mfp)}} \frac{\log x}{\log N(\mfq)} \right) \\
    & \ll_h (\log x)^2 \sum_{\substack{N(\mfp) \leq x^{1/h}/2}} \frac{1}{N(\mfp) (\log^2 (x^{1/h}/N(\mfp))) (\log N(\mfp))} \\
    & \ll_h 1.
\end{align*}
For $X = \{ q^z : z \in \mathbb{Z} \}$, similar bounds for $J_1$ and $J_2$ can be proved using $x^{1/h}/q$ instead of $x^{1/h}/2$.

Combining the last three results, we obtain
\begin{align}\label{hfullp4Omega}
    & \kappa \gamma_{h} x^{1/h} \sum_{\substack{\mfp,\mfq \\ \mfp \neq \mfq, N(\mfp) N(\mfq) \leq x^{1/h}}} \bigg( \sum_{k=h}^{\lfloor \frac{\log x}{\log N(\mfp)} \rfloor} \sum_{l=h}^{\lfloor \frac{\log x}{\log N(\mfq)} \rfloor} \frac{k}{N(\mfp)^{k/h} \left( 1 + \frac{N(\mfp)^{-1}}{1 - N(\mfp)^{-1/h}} \right)}  \cdot \notag \\
    & \hspace{1cm} \frac{l}{N(\mfq)^{l/h}  \left( 1 + \frac{N(\mfq)^{-1}}{1 - N(\mfq)^{-1/h}} \right)} \bigg) \notag \\
    & = \kappa \gamma_{h} x^{1/h} \sum_{\substack{\mfp,\mfq \\ \mfp \neq \mfq, N(\mfp) N(\mfq) \leq x^{1/h}}}  \bigg( \frac{h(1- N(\mfp)^{-1/h})+N(\mfp)^{-1/h}}{N(\mfp)(1- N(\mfp)^{-1/h}) (1- N(\mfp)^{-1/h}+ N(\mfp)^{-1})}  \cdot \notag \\
    & \hspace{1cm} \frac{h(1- N(\mfq)^{-1/h})+N(\mfq)^{-1/h}}{N(\mfq)(1- N(\mfq)^{-1/h}) (1- N(\mfq)^{-1/h}+ N(\mfq)^{-1})} \bigg) + O_h \left( \frac{x^{1/h} \log \log x}{\log x} \right).
\end{align}
Thus, to complete the proof, we only need to estimate the main term in \eqref{hfullp4Omega}. Notice that
\begin{align*}
     & \sum_{\substack{\mfp,\mfq \\ \mfp \neq \mfq, N(\mfp) N(\mfq) \leq x^{1/h}}}  \frac{h(1- N(\mfp)^{-1/h})+N(\mfp)^{-1/h}}{N(\mfp)(1- N(\mfp)^{-1/h}) (1- N(\mfp)^{-1/h}+ N(\mfp)^{-1})}  \cdot \notag \\
    & \hspace{1cm} \frac{h(1- N(\mfq)^{-1/h})+N(\mfq)^{-1/h}}{N(\mfq)(1- N(\mfq)^{-1/h}) (1- N(\mfq)^{-1/h}+ N(\mfq)^{-1})} \notag \\
    & = \sum_{\substack{\mfp,\mfq \\ N(\mfp) N(\mfq) \leq x^{1/h}}} \frac{h(1- N(\mfp)^{-1/h})+N(\mfp)^{-1/h}}{N(\mfp)(1- N(\mfp)^{-1/h}) (1- N(\mfp)^{-1/h}+ N(\mfp)^{-1})}  \notag \\
    & \hspace{1cm} \frac{h(1- N(\mfq)^{-1/h})+N(\mfq)^{-1/h}}{N(\mfq)(1- N(\mfq)^{-1/h}) (1- N(\mfq)^{-1/h}+ N(\mfq)^{-1})} \bigg) \notag \\
    & \hspace{.5cm} - \sum_{\substack{N(\mfp) \leq x^{1/(2h)}}} \left( \frac{h(1- N(\mfp)^{-1/h})+N(\mfp)^{-1/h}}{N(\mfp)(1- N(\mfp)^{-1/h}) (1- N(\mfp)^{-1/h}+ N(\mfp)^{-1})} \right)^2.
\end{align*}
Thus, using part 4 of \lmaref{boundnm}, we have
\begin{align}\label{hfullp5Omega}
    & \sum_{\substack{\mfp,\mfq \\ \mfp \neq \mfq, N(\mfp) N(\mfq) \leq x^{1/h}}}  \frac{h(1- N(\mfp)^{-1/h})+N(\mfp)^{-1/h}}{N(\mfp)(1- N(\mfp)^{-1/h}) (1- N(\mfp)^{-1/h}+ N(\mfp)^{-1})}  \cdot \notag \\
    & \hspace{1cm} \frac{h(1- N(\mfq)^{-1/h})+N(\mfq)^{-1/h}}{N(\mfq)(1- N(\mfq)^{-1/h}) (1- N(\mfq)^{-1/h}+ N(\mfq)^{-1})} \notag \\
    & = \sum_{\substack{\mfp,\mfq \\ N(\mfp) N(\mfq) \leq x^{1/h}}}  \frac{h(1- N(\mfp)^{-1/h})+N(\mfp)^{-1/h}}{N(\mfp)(1- N(\mfp)^{-1/h}) (1- N(\mfp)^{-1/h}+ N(\mfp)^{-1})}  \cdot \notag \\
    & \hspace{1cm} \frac{h(1- N(\mfq)^{-1/h})+N(\mfq)^{-1/h}}{N(\mfq)(1- N(\mfq)^{-1/h}) (1- N(\mfq)^{-1/h}+ N(\mfq)^{-1})} \notag \\
    & \hspace{.5cm} 
    % - \sum_{N(\mfp)} \left( \frac{h(1- N(\mfp)^{-1/h})+N(\mfp)^{-1/h}}{N(\mfp)(1- N(\mfp)^{-1/h}) (1- N(\mfp)^{-1/h}+ N(\mfp)^{-1})} \right)^2 
    - \sum_\mfp \left( \frac{h(N(\mfp)^{1/h}-1)+1}{(N(\mfp)^{1/h}-1) (N(\mfp) - N(\mfp)^{1-1/h} +1)} \right)^2 + O_h \left( \frac{1}{x^{1/(2h)} \log x} \right).
\end{align}
Now, using
\begin{align*}
    & \frac{h(1- N(\mfp)^{-1/h})+N(\mfp)^{-1/h}}{N(\mfp)(1- N(\mfp)^{-1/h}) (1- N(\mfp)^{-1/h}+ N(\mfp)^{-1})} \\
    & = \frac{h}{N(\mfp)} + \frac{h(N(\mfp)-N(\mfp)^{1-(1/h)}-N(\mfp)^{1/h} + 1)+N(\mfp)}{N(\mfp)(N(\mfp)^{1/h}-1)(N(\mfp) - N(\mfp)^{1-(1/h)} + 1)},
\end{align*}
a similar result for another prime $\mfq$, and the symmetry of primes $\mfp$ and $\mfq$, we write the first sum on the right-hand side of \eqref{hfullp5Omega} as
\begin{align*}
      & \sum_{\substack{\mfp,\mfq \\ N(\mfp) N(\mfq) \leq x^{1/h}}}  \frac{h(1- N(\mfp)^{-1/h})+N(\mfp)^{-1/h}}{N(\mfp)(1- N(\mfp)^{-1/h}) (1- N(\mfp)^{-1/h}+ N(\mfp)^{-1})}  \cdot \notag \\
    & \hspace{1cm} \frac{h(1- N(\mfq)^{-1/h})+N(\mfq)^{-1/h}}{N(\mfq)(1- N(\mfq)^{-1/h}) (1- N(\mfq)^{-1/h}+ N(\mfq)^{-1})} \notag \\
    & = h^2 \sum_{\substack{\mfp,\mfq \\ N(\mfp) N(\mfq) \leq x^{1/h}}} \frac{1}{N(\mfp) N(\mfq)} \\
    & \hspace{.5cm} + 2 h \sum_{\substack{\mfp,\mfq \\ N(\mfp) N(\mfq) \leq x^{1/h}}} \frac{1}{N(\mfp)} \left( \frac{h(N(\mfq)-N(\mfq)^{1-(1/h)}-N(\mfq)^{1/h} + 1)+N(\mfq)}{N(\mfq)(N(\mfq)^{1/h}-1)(N(\mfq) - N(\mfq)^{1-(1/h)} + 1)} \right)\\
    & \hspace{.5cm} + \sum_{\substack{\mfp,\mfq \\ N(\mfp) N(\mfq) \leq x^{1/h}}} \left( \frac{h(N(\mfp)-N(\mfp)^{1-(1/h)}-N(\mfp)^{1/h} + 1)+N(\mfp)}{N(\mfp)(N(\mfp)^{1/h}-1)(N(\mfp) - N(\mfp)^{1-(1/h)} + 1)} \right) \cdot \\
    & \hspace{1cm} \left( \frac{h(N(\mfq)-N(\mfq)^{1-(1/h)}-N(\mfq)^{1/h} + 1)+N(\mfq)}{N(\mfq)(N(\mfq)^{1/h}-1)(N(\mfq) - N(\mfq)^{1-(1/h)} + 1)} \right).
\end{align*}
The first sum on the right-hand side above is estimated using \lmaref{gsaidak}. For the second sum, when $X = \mathbb{Q}$, we use parts 1, 3, and 4 of \lmaref{boundnm}, and $\log(x^{1/h}/N(\mfp)) \geq \log 2$ together to obtain
\begin{align*}
    & \sum_{\substack{\mfp,\mfq \\ N(\mfp) N(\mfq) \leq x^{1/h}}} \frac{1}{N(\mfp)} \left( \frac{h(N(\mfq)-N(\mfq)^{1-(1/h)}-N(\mfq)^{1/h} + 1)+N(\mfq)}{N(\mfq)(N(\mfq)^{1/h}-1)(N(\mfq) - N(\mfq)^{1-(1/h)} + 1)} \right) \\
    & = \sum_{\substack{N(\mfp) \leq x^{1/h}/2}} \frac{1}{N(\mfp)} \sum_{\substack{N(\mfq) \leq x^{1/h}/N(\mfp)}} \frac{h(N(\mfq)-N(\mfq)^{1-(1/h)}-N(\mfq)^{1/h} + 1)+N(\mfq)}{N(\mfq)(N(\mfq)^{1/h}-1)(N(\mfq) - N(\mfq)^{1-(1/h)} + 1)} \\
    & = \sum_{\substack{N(\mfp) \leq x^{1/h}/2}} \frac{1}{N(\mfp)} \bigg( \sum_\mfp \frac{h(N(\mfp)-N(\mfp)^{1-(1/h)}-N(\mfp)^{1/h} + 1)+N(\mfp)}{N(\mfp)(N(\mfp)^{1/h}-1)(N(\mfp) - N(\mfp)^{1-(1/h)} + 1)} \\
    & \hspace{.5cm} +  O_h \left( \frac{N(\mfp)^{1/h}}{x^{1/h^2} \log(x^{1/h}/N(\mfp))} \right) \bigg) \\
    & = \left( \sum_\mfp \frac{h(N(\mfp)-N(\mfp)^{1-(1/h)}-N(\mfp)^{1/h} + 1)+N(\mfp)}{N(\mfp)(N(\mfp)^{1/h}-1)(N(\mfp) - N(\mfp)^{1-(1/h)} + 1)}  \right) \left(  \log \log x + \mfa - \log h \right) + O_h \left( \frac{1}{\log x} \right).
\end{align*}
Similarly, for the third sum, we obtain 
\begin{align*}
    & \sum_{\substack{\mfp,\mfq \\ N(\mfp) N(\mfq) \leq x^{1/h}}} \frac{h(N(\mfp)-N(\mfp)^{1-(1/h)}-N(\mfp)^{1/h} + 1)+N(\mfp)}{N(\mfp)(N(\mfp)^{1/h}-1)(N(\mfp) - N(\mfp)^{1-(1/h)} + 1)} \cdot \\
    & \hspace{1cm} \frac{h(N(\mfq)-N(\mfq)^{1-(1/h)}-N(\mfq)^{1/h} + 1)+N(\mfq)}{N(\mfq)(N(\mfq)^{1/h}-1)(N(\mfq) - N(\mfq)^{1-(1/h)} + 1)}  \\
    & = \sum_{\substack{N(\mfp) \leq x^{1/h}/2}} \frac{h(N(\mfp)-N(\mfp)^{1-(1/h)}-N(\mfp)^{1/h} + 1)+N(\mfp)}{N(\mfp)(N(\mfp)^{1/h}-1)(N(\mfp) - N(\mfp)^{1-(1/h)} + 1)} \cdot \\
    & \hspace{1cm} \sum_{\substack{N(\mfq) \leq x^{1/h}/N(\mfp)}} \frac{h(N(\mfq)-N(\mfq)^{1-(1/h)}-N(\mfq)^{1/h} + 1)+N(\mfq)}{N(\mfq)(N(\mfq)^{1/h}-1)(N(\mfq) - N(\mfq)^{1-(1/h)} + 1)} \\
    % & = \sum_{\substack{N(\mfp) \\ N(\mfp) \leq x^{1/h}/2}}  \frac{h(N(\mfp)-N(\mfp)^{1-(1/h)}-N(\mfp)^{1/h} + 1)}{N(\mfp)(N(\mfp)^{1/h}-1)(N(\mfp) - N(\mfp)^{1-(1/h)} + 1)}  \Bigg( \sum_\mfp \frac{h(N(\mfp)-N(\mfp)^{1-(1/h)}-N(\mfp)^{1/h} + 1)}{N(\mfp)(N(\mfp)^{1/h}-1)(N(\mfp) - N(\mfp)^{1-(1/h)} + 1)} \\
    % & \hspace{.5cm} +  O_h \left( \frac{N(\mfp)^{1/h}}{x^{1/h^2} \log(x^{1/h}/N(\mfp))} \right) \Bigg) \\
    & =  \left(\sum_\mfp \frac{h(N(\mfp)-N(\mfp)^{1-(1/h)}-N(\mfp)^{1/h} + 1)+N(\mfp)}{N(\mfp)(N(\mfp)^{1/h}-1)(N(\mfp) - N(\mfp)^{1-(1/h)} + 1)} \right)^2 + O_h \left( \frac{\log \log x}{x^{1/h^2} \log x} \right).
\end{align*}
For $X = \{ q^z : z \in \mathbb{Z} \}$, similar results as the above two can be proved using $x^{1/h}/q$ instead of $x^{1/h}/2$, and the inequality $\log(x^{1/h}/N(\mfp)) \geq \log q$. 

Combining the last three results with \lmaref{gsaidak} and simplifying further using \eqref{simplification}, and using \eqref{def-D3} and \eqref{def-D4}, we obtain
\begin{align*}
     & \sum_{\substack{\mfp,\mfq \\ N(\mfp) N(\mfq) \leq x^{1/h}}}  \frac{h(1- N(\mfp)^{-1/h})+N(\mfp)^{-1/h}}{N(\mfp)(1- N(\mfp)^{-1/h}) (1- N(\mfp)^{-1/h}+ N(\mfp)^{-1})}  \cdot \notag \\
    & \hspace{1cm} \frac{h(1- N(\mfq)^{-1/h})+N(\mfq)^{-1/h}}{N(\mfq)(1- N(\mfq)^{-1/h}) (1- N(\mfq)^{-1/h}+ N(\mfq)^{-1})} \notag \\
     & = h^2\left( (\log \log x - \log h)^2 + 2 \mfa (\log \log x - \log h) + \mfa^2 + \mathfrak{B} \right) \notag \\
     & + 2 h \left( \sum_\mfp \frac{h(N(\mfp)-N(\mfp)^{1-(1/h)}-N(\mfp)^{1/h} + 1)+N(\mfp)}{N(\mfp)(N(\mfp)^{1/h}-1)(N(\mfp) - N(\mfp)^{1-(1/h)} + 1)}  \right) \left(  \log \log x + \mfa - \log h \right) \notag \\
     & + \left(\sum_\mfp \frac{h(N(\mfp)-N(\mfp)^{1-(1/h)}-N(\mfp)^{1/h} + 1)+N(\mfp)}{N(\mfp)(N(\mfp)^{1/h}-1)(N(\mfp) - N(\mfp)^{1-(1/h)} + 1)} \right)^2  + O_h \left( \frac{\log \log x}{\log x} \right) \\
     & = h^2 (\log \log x)^2 + 2 \mathfrak{D}_3 h \log \log x + \mathfrak{D}_3^2 + h^2 \mathfrak{B}
     % & \hspace{.5cm} + 2 \left( \sum_\mfp \frac{h(N(\mfp)-N(\mfp)^{1-(1/h)}-N(\mfp)^{1/h} + 1)}{N(\mfp)(N(\mfp)^{1/h}-1)(N(\mfp) - N(\mfp)^{1-(1/h)} + 1)}  \right) \left(  \log \log x + B_1 - \log h \right)  + (\log h)^2 - 2 (\log h) \sum_\mfp \frac{N(\mfp)^{-1/h} - N(\mfp)^{-1}}{N(\mfp)(1-N(\mfp)^{-1/h}+N(\mfp)^{-1})} \\
     % & \hspace{1cm}  + \left( \sum_\mfp \frac{N(\mfp)^{-1/h} - N(\mfp)^{-1}}{N(\mfp)(1-N(\mfp)^{-1/h}+N(\mfp)^{-1})} \right)^2 
     + O_h \left( \frac{\log \log x}{\log x} \right).
\end{align*}
Combining the above with \eqref{hfullp1Omega}, \eqref{need_later-D4}, \eqref{hfullp2Omega}, \eqref{hfullp3Omega}, \eqref{hfullp4Omega}, and \eqref{hfullp5Omega}, we obtain
\begin{align*}
     \sum_{\substack{\mfm \in \mathcal{N}_h(x)}} \Omega^2(\mfm)
     & = h^2 \kappa \gamma_{h} x^{1/h} (\log \log x)^2 + (2\mathfrak{D}_3 + h) h \kappa \gamma_{h} x^{1/h} \log \log x + \mathfrak{D}_4 \kappa \gamma_{h} x^{1/h} \\
     % +   \left( 2B_1 + 2 \sum_\mfp \frac{N(\mfp)^{-1/h} - N(\mfp)^{-1}}{N(\mfp)(1-N(\mfp)^{-1/h}+N(\mfp)^{-1})} - 2\log h + 1 \right) \gamma_{0,h} x^{1/h} \log \log x \\
     % & \hspace{.5cm} + \Bigg( B_1^2 - \zeta(2) + B_1 \left( 2 \sum_\mfp \frac{N(\mfp)^{-1/h} - N(\mfp)^{-1}}{N(\mfp)(1-N(\mfp)^{-1/h}+N(\mfp)^{-1})} - 2 \log h + 1 \right) + (\log h)^2  - \log h  \\
     % & \hspace{.5cm} + \sum_{N(\mfp)}  \frac{N(\mfp)^{-1/h} - N(\mfp)^{-1}}{ N(\mfp) \left( 1- N(\mfp)^{-1/h}+N(\mfp)^{-1} \right) } - 2 (\log h) \sum_\mfp \frac{N(\mfp)^{-1/h} - N(\mfp)^{-1}}{N(\mfp)(1-N(\mfp)^{-1/h}+N(\mfp)^{-1})} + \left( \sum_\mfp \frac{N(\mfp)^{-1/h} - N(\mfp)^{-1}}{N(\mfp)(1-N(\mfp)^{-1/h}+N(\mfp)^{-1})} \right)^2  \\
     % & \hspace{.5cm}  - \sum_{N(\mfp)} \left( \frac{1}{N(\mfp)(1-N(\mfp)^{-1/h}+N(\mfp)^{-1})} \right)^2 \Bigg) \gamma_{0,h} x^{1/h}
     & \hspace{.5cm} + O_h \left( \frac{x^{1/h} \log \log x}{\log x} \right).
\end{align*}
This completes the proof of \thmref{hfullOmega}. 
\end{proof}

In their previous work \cite{dkl5}, the authors studied the generalized function 
$$f(\mfm) = \sum_{\substack{\mfp \\ n_\mfp(\mfm) \geq 1}} \nu_\mfp(\mfm),$$ 
where the weight function was $\nu_\mfp(\mfm) = 1$. This article extends that research by exploring a new weight function, $\nu_\mfp(\mfm) = n_\mfp(\mfm)$, which represents the multiplicity of $\mfp$ in $\mfm$. The authors plan to investigate other generalized functions with different weight functions in a future article.

\bibliographystyle{plain} % We choose the "plain" reference style
\bibliography{mybib.bib} % Entries are in the biblio.bib file

\end{document}